\documentclass[leqno,11pt]{amsart}
\usepackage{amsmath}
\usepackage{amsfonts}
\usepackage{amssymb}
\usepackage{amsthm}

\usepackage{mathrsfs}
\usepackage{dsfont}
\usepackage{txfonts}
\usepackage{mathabx}

\usepackage{graphicx}
\usepackage[all]{xy}
\usepackage{pgf,tikz}
\usetikzlibrary{arrows}
\usepackage{subfigure}

\usepackage[pdftex]{hyperref}

\usepackage{color}
\usepackage[shortlabels]{enumitem}

\newtheorem{theorem}{Theorem}[section]

\newtheorem{corollary}[theorem]{Corollary}

\newtheorem{proposition}[theorem]{Proposition}

\theoremstyle{definition}

\newtheorem{notation}[theorem]{Notation}

\theoremstyle{remark}
\newtheorem{remark}[theorem]{Remark}

\newcommand{\R}{\mathbb{R}}

\newcommand{\bd}{\partial}
\newcommand{\rn}{\mathbb{R}^n}
\newcommand{\rno}{\mathbb{R}^{n+1}}

\newcommand{\Ak}{|A|^2_k}
\newcommand{\cL}{\mathcal{L}}

\newcommand{\cv}{\text{Conv}}

\newcommand{\cM}{\mathcal{M}}
\newcommand{\cK}{\mathcal{K}}
\newcommand{\cA}{\mathcal{A}}

\newcommand{\Sj}{\Gamma}
\newcommand{\grad}{\upsilon}
\newcommand{\tN}{\widetilde{\nabla}}

\newcommand{\subscript}[2]{$#1 _ #2$}

\title{The   $Q_k$ flow on complete non-compact graphs}

\author{Kyeongsu Choi}
\address{ {\bf Kyeongsu Choi:} Department of Mathematics, Columbia University, 2990 Broadway, New York, NY 10027, USA.}
\email{kschoi@math.columbia.edu}
\author{Panagiota Daskalopoulos}
\address{ {\bf P. Daskalopoulos:} Department of Mathematics, Columbia University, 2990 Broadway, New York, NY 10027, USA.}
\email{pdaskalo@math.columbia.edu}

\begin{document}

\maketitle

\section{Introduction}

We study in this work the long time existence of a family of complete non-compact strictly convex hypersurfaces $\Sigma_t$
 embedded in $\rno$ which  evolve  by  the $Q_k$-flow. 
Given a complete and convex hypersurface $\Sigma_0$ embedded in $\rno$, we assume that $F_0:M^n \rightarrow \rno$ is an immersion  with $F_0(M^n)=\Sigma_0$. We say that the  one-parameter family of immersions $$F:M^n\times [0,T) \rightarrow \rno$$ is a solution of the $Q_k$-flow $(1\leq k \leq n)$,   if $F(M^n,t)=\Sigma_t$ are complete convex hypersurfaces for all $t\in [0,T)$  and $F(\cdot,t)$ satisfies 
\begin{equation}\label{eq:INT Qkflow}
\begin{cases} 
{\displaystyle  \frac{\bd}{\bd t}} F(p,t) & = \; Q_k (p,t) \vec{n}  (p,t) \\
 F(p , 0 ) & = \; F_0(p). 
\end{cases} \tag{$*_k^n$}
\end{equation}
where   $\vec{n}(p,t)$ is the unit normal vector pointing inside the convex hull of $\Sigma_t$.  The speed 
$$\displaystyle Q_k (p,t) := \frac{S_k(p,t)}{S_{k-1}(p,t)}$$  is the quotient of the  elementary successive polynomials  of the principal curvatures $\{\lambda_1(p,t) , \cdots , \lambda_n(p,t)\} $ of $\Sigma_t$ at $F(p,t)$, given by 
\begin{equation*}
S_0 (p,t)   = 1, \qquad  S_k (p,t)  = \sum_{ 1\leq i_1 < \cdots < i_k \leq n } \lambda_{i_1}(p,t)\cdots\lambda_{i_k}(p,t) \qquad \text{for} \; 1 \leq k \leq n.
\end{equation*}

In \cite{A94NonlinearFlow}, B. Andrews showed  the existence of strictly convex closed solutions of a class of nonlinear flow which includes the $Q_k$-flow. S. Diater extended the results to closed convex solutions with the positive $S_{k-1}$ curvature in \cite{Dieter05Qk}. Moreover, Caputo, Daskalopoulos, and Sesum showed the existence of compact convex $C^{1,1}$ viscosity solutions with flat sides in \cite{CD09HMCF} and in \cite{CDN10Qk}. Closed non-convex solutions of the $Q_2$-flow in $\mathbb{R}^3$, the  \textit{Harmonic mean curvature flow},  were considered by Daskalopoulos and  Hamilton  in \cite{DH06HMCF} and  by Daskalopoulos, Hamilton and Sesum in \cite{DN10HMCF}.

\smallskip

The equation  \eqref{eq:INT Qkflow} is  fully-nonlinear except from the case of $k=1$ which is the flow by Mean curvature.  
The evolution of {\em entire graphs}  by the Mean curvature flow was studied by Ecker and G. Huisken in  \cite{HE89MCFasymptotic, HE91MCFexist}.
More recently,  S{\'a}ez  and Schn{\"u}rer \cite{OS14CpltMean}  showed   the existence of  complete solutions of the \textit{Mean curvature flow} 
for an  initial hypersurface which is a graph $\Sigma_0=\{(x,u_0(x)):x\in \Omega_0\}$ over a bounded domain  $\Omega_0$,  and $u_0(x) \to +\infty$ as $x \to \bd \Omega_0$.

The Ecker and Huisken result in  \cite{HE91MCFexist}  shows   that in some sense the Mean curvature flow  
behaves better than the  heat equation on $\rn$, namely an entire graph  
solution exists for  all time independently from the growth of the initial surface at infinity. 
The initial entire graph is assumed to be only locally Lipschitz.  
This result is based on a   local gradient estimate which is then  combined with the evolution 
of the norm of the second fundamental form $|A|^2$  to give a  local bound  on $|A|^2$,  which is independent from the behavior of the
solution at  spatial infinity. 
The latter is achieved  by adopting the  well known technique 
of   Caffarelli, Nirenberg and Spruck  in \cite{CNS88trick} in this geometric setting.

\smallskip 
An  open question between the experts in the field is whether the techniques of Ecker and Huisken in
\cite{HE89MCFasymptotic, HE91MCFexist} can be extended to  the fully-nonlinear setting and in particular on complete  convex graphs 
evolving  by the {\em  $Q_k$-flow }. In this work we will show that this
the case under the weak  convexity assumption, as our main result stated next shows.

\begin{theorem} \label{thm:INT Existence}
Let   $\Sigma_0=\{ (x,  u_0(x)): x \in \Omega\}$ be a smooth weakly  convex graph defined by a  function $ u_0:\Omega \rightarrow \R$ 
on a convex open domain $\Omega \subset \rn$ such that  
\begin{enumerate}
\item[{\em (i)}]  $u_0 $ attains its minimum in $\Omega_0$, and $\inf_{\Omega_0}u_0 \geq 0$.
\item[{\em (ii)}] If $ \Omega_0 \neq \rn$,  for all $x_0 \in \bd \Omega_0$, $\displaystyle \lim_{x \rightarrow x_0} u_0(x) = +\infty$ holds.
\item[{\em (iii)}]  If $\Omega_0$ is unbounded, then $\displaystyle \lim_{|x| \rightarrow +\infty} u_0(x) = +\infty$ holds.
\item[{\em (iv)}]  $\Sigma_0$ has the positive $Q_k$ curvature at all $(x,u_0(x)) \in \Sigma_0$.
\end{enumerate}
Then, given a smooth immersion $F_0$ of $F_0(M^n)=\Sigma_0$, there is a complete convex solution $\Sigma_t$ of \emph{(\ref{eq:INT Qkflow})} for $$T \geq \frac{k}{2(n-k+1)}R^2.$$ Moreover, for each $t \in [0,T)$,  $\Sigma_t$ is a graph.  In particular, 
if  for all $R >0$ there exists a ball    $B_R(x_R) \subset \Omega_0$\;, then \emph{(\ref{eq:INT Qkflow})} has an all-time existing $\Sigma_t$ solution.
\end{theorem}

\begin{figure}[h]
\subfigure[$\Omega=\rn$]{
\begin{tikzpicture}[line cap=round,line join=round,>=triangle 45,x=0.18cm,y=0.18cm]
\clip(-11.6278237129,-0.532800819239) rectangle (11.6020914256,16.4171110141);
\draw [samples=50,domain=-0.99:0.99,rotate around={90.:(0.,0.)},xshift=0.cm,yshift=0.cm] plot ({1.28919490362*(1+(\x)^2)/(1-(\x)^2)},{1.52904430952*2*(\x)/(1-(\x)^2)});
\draw [samples=50,domain=-0.99:0.99,rotate around={90.:(0.,0.)},xshift=0.cm,yshift=0.cm] plot ({1.28919490362*(-1-(\x)^2)/(1-(\x)^2)},{1.52904430952*(-2)*(\x)/(1-(\x)^2)});
\begin{scriptsize}
\draw[color=black] (-6,7) node[scale=1.4] {$\Sigma_0$};
\end{scriptsize}
\end{tikzpicture}
}
\subfigure[$\Omega = B_R(0)$]{
\begin{tikzpicture}[line cap=round,line join=round,>=triangle 45,x=0.5cm,y=0.5cm]
\clip(-4.67471793468,-1.) rectangle (4.74524982753,5.31432920204);
\draw [rotate around={90.:(0.,5.)}] (0.,5.) ellipse (2.91547594742cm and 1.4cm);
\draw [dash pattern=on 1pt off 1pt] (3.05,-1.) -- (3.05,5.31432920204);
\draw [dash pattern=on 1pt off 1pt] (-3.05,-1.) -- (-3.05,5.31432920204);
\begin{scriptsize}
\draw[color=black] (-2,3) node[scale=1.4] {$\Sigma_0$};
\end{scriptsize}
\end{tikzpicture}
}
\subfigure[$\Omega = \mathbb{R}^{n-1}\times \mathbb{R}^+$]{
\begin{tikzpicture}[line cap=round,line join=round,>=triangle 45,x=0.3cm,y=0.3cm]
\clip(-2.22055421895,1.36898033791) rectangle (11.0682860864,12);
\draw [samples=50,domain=-0.99:0.99,rotate around={56.309932474:(0.,0.)},xshift=0.cm,yshift=0.cm] plot ({2.9868607644*(1+(\x)^2)/(1-(\x)^2)},{2.01956994781*2*(\x)/(1-(\x)^2)});
\draw [samples=50,domain=-0.99:0.99,rotate around={56.309932474:(0.,0.)},xshift=0.cm,yshift=0.cm] plot ({2.9868607644*(-1-(\x)^2)/(1-(\x)^2)},{2.01956994781*(-2)*(\x)/(1-(\x)^2)});
\draw [dash pattern=on 4pt off 4pt] (0.,1.36898033791) -- (0.,12);
\begin{scriptsize}
\draw[color=black] (1.5,8) node[scale=1.4] {$\Sigma_0$};
\end{scriptsize}
\end{tikzpicture}
}
\caption{Examples of the initial hypersurface $\Sigma_0$}
\end{figure}
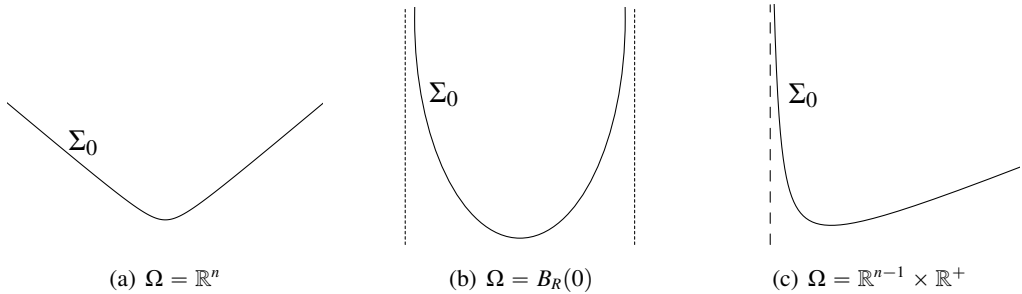

\begin{remark}[General initial data]
Given a complete and strictly convex hypersurface $\Sigma^n_0$  embedded in $\rno$, there exist an orthogonal matrix $A \in O(n+1)$, a vector $Y_0 \in \rno$ and a function $u_0: \Omega \rightarrow \mathbb{R}$ such that $A\Sigma_0+Y_0 \coloneqq \{ Y_0+AX : X \in \Sigma_0 \}=\{ (x,u_0(x)):x \in \Omega\}$, and  the conditions (i), (ii), (iii) above hold. Thus, Theorem \ref{thm:INT Existence} shows the existence of a complete convex solution $\Sigma_t$ of \eqref{eq:INT Qkflow} for any complete smooth strictly convex hypersurface $\Sigma_0$.
\end{remark}

\noindent{\em Discussion of the proof of Theorem \ref{thm:INT Existence}}:  The proof of Theorem \ref{thm:INT Existence} 
mainly relies on three  a'priori local estimates:  the local gradient bound shown in Theorem \ref{thm:Pre Gradient},
the local speed estimate given in Theorem \ref{thm:SE Speed Estimate}
and a local bound  from above on the 
second fundamental form $|A|^2$ given in Theorem \ref{thm:CE Local Pinching Estimate}. The gradient and the speed estimates  use the well known technique by  Caffarelli, Nirenberg and Spruck in \cite{CNS88trick}  also used by Esker and Huisken 
in the context of the Mean curvature flow in  \cite{HE91MCFexist}. Then, by using the concavity of the $Q_k(\lambda)$ function, we derive 
a local bound on  $|A|^2$  by modifying the elliptic estimate
by W. Sheng, J. Urbas  and X.-J. Wang in \cite{SUX04Qk} to the parabolic setting.
We also establish  derivative curvature estimates of any order,  by adopting the  technique of  Shi's local derivative estimate 
to the setting of  concave fully nonlinear equations. The long time existence is shown  by approximation with compact hypersurfaces
and applying the local a priori estimates.

\begin{notation}\label{not:INT notation}
We summarize the following notation, which will be frequently used in this paper.
\begin{enumerate}
\item We recall the \textit{second fundamental form} $h_{ij} \coloneq \langle \nabla_i \nabla_j F, \vec{n} \rangle$ and the \textit{metric} $g_{ij} \coloneq \langle F_i, F_j \rangle$, where $F_i \coloneq \nabla_i F$.
\item We denote by $\bar u:M^n\rightarrow\mathbb{R}$ the \textit{height function}
$\bar u(p,t)\coloneq \langle F(p,t),\vec{e}_{n+1} \rangle$. Also, given a constant $M \in \mathbb{R}$, we define a \textit{cut-off function} $\psi$ by $$\psi(p,t)\coloneq (M-\bar u(p,t))_+ =\max (M- \bar u , 0).$$
\item  $\grad \coloneq \langle \vec{n},\vec{e}_{n+1} \rangle^{-1}$ denote the \textit{gradient function} (as in \cite{HE91MCFexist}). 
\item We denote by $\cL$\, the \textit{linearized} operator, $$\cL \coloneq \frac{\bd Q_k}{\bd h_{ij}} \nabla_i \nabla_j.  $$ In addition, $\langle \, , \, \rangle_\cL$ denotes the \textit{inner product} $\displaystyle \langle \nabla f,\nabla g \rangle_\cL = \frac{\bd Q_k}{\bd h_{ij}} \nabla_i f \nabla_j g  $, where $f,g$ are differentiable functions on $M^n$, and $\|\cdot \|_{\cL}$ denotes the $\cL$-\textit{norm} given by the inner product $\langle \,,\,\rangle_{\cL}$.
\item We recall \textit{the square sum of the principal curvatures} $|A|^2=h_{ij}h^{ij}$ and \textit{the $m$-th order derivative of curvature} $$|\nabla^m A|^2 = \nabla_{i_1}\cdots\nabla_{i_m}h_{jk}\nabla^{i_1}\cdots\nabla^{i_m}h^{jk}.$$
\item For the \textit{principal curvatures} $\{\lambda_1,\cdots,\lambda_n \}$, we denote by $\lambda_{\max}$ the \textit{largest principal curvature} $\lambda_{\max} \coloneq \max \{\lambda_1 ,\cdots,\lambda_n\}$. Also, denote  the following  functions of the principal curvatures
\begin{align*}
S_{k;i}(\lambda) \coloneq \frac{\bd S_{k+1}(\lambda)}{\bd \lambda_i},  \quad  \Ak(\lambda) \coloneq \sum_{i=1}^n\frac{\bd Q_k(\lambda)}{\bd \lambda_i} \lambda_i^2, \quad D_iQ_k=\frac{\bd Q_k(\lambda)}{\bd \lambda_i},  \quad D_{ij}Q_k=\frac{\bd^2 Q_k(\lambda)}{\bd \lambda_t \bd\lambda_j}.
\end{align*}
\end{enumerate}
\end{notation}

\section{Preliminaries}\label{sec-prelim}
In this section, we will  review some  properties of the symmetric function $Q_k(\lambda)$ of $\lambda$, and we will derive some basic evolution equations under the $Q_k$-flow.  We will 
also establish a  local gradient estimate  and a  local lower bound on  the speed, as straightforward consequence of  the evolution equations.

\begin{proposition}\label{prop:Pre 3rd Order}
Assume that $\Sigma$ is a convex smooth hypersurface in $\rno$ with the positive $Q_k$ curvature and $F:M^n \rightarrow \rno$ is  a smooth immersion satisfying $F(M^n)=\Sigma$. Let us choose an orthonormal frame at some point $F(p)$ satisfying $g_{ij}(p)=\delta_{ij}$, $ h_{ij}(p)=\delta_{ij}\lambda_i(p)$. Then,  the following holds at $F(p)$  for each $m\in \{1,\cdots,n\}$
\begin{align*}
\sum_{i,j,p,q}\frac{\bd^2 Q_k}{\bd h_{ij} \bd h_{pq}}\nabla^m h_{ij} \nabla_m h_{pq}  
\leq 2 \sum_{i < j} \frac{D_i Q_k-D_j Q_k}{\lambda_i -\lambda_j}\nabla^m h_{ij} \nabla_m h_{ij}. 
\end{align*}
\end{proposition}

\begin{proof}
We recall the following identity which holds on homogeneous of degree one functions of matrices, given in  \cite{A94NonlinearFlow} (see also in  \cite{CDN10Qk}):
\begin{align*}
\sum_{i,j,p,q}\frac{\bd^2 Q_k}{\bd h_{ij} \bd h_{pq}}\nabla^m h_{ij} \nabla_m h_{pq} =\sum_{i,j}D_{ij}Q_k\nabla^m h_{ii} \nabla_m h_{jj} +  \sum_{i \neq j} \frac{D_i Q_k-D_j Q_k}{\lambda_i -\lambda_j}\nabla^m h_{ij} \nabla_m h_{ji}. 
\end{align*}
By the  concavity of $Q_k(\lambda)$ we have   $\displaystyle \sum_{i,j} D_{ij} Q_k \nabla^m h_{ii} \nabla_m h_{jj} \leq 0$, hence the desired  inequality follows.  
\end{proof}

\begin{proposition}
If  $Q_k(\lambda_1,\cdots,\lambda_n) >0$ and $\lambda_i \geq 0$\, for all $i \in \{1,\cdots,n \}$,  then the following hold 
\begin{align*}
 n^{-2} Q_k^2 \, \lambda_{\max}^{-2} & \leq D_i Q_k \leq 1 \tag{2.1}\label{eq:Pre Uniform Parabolicity} \\
D_i Q_k &\leq  \lambda_i^{-2}Q_k^{2}\tag{2.2} \label{eq:Pre Parabolic Coefficients}\\
\frac{k}{n-k+1} \, Q_k^2 &\leq \Ak  \leq n Q_k^2
\label{eq:Pre Reaction Inequality}\tag{2.3}
\end{align*}
\end{proposition}

\begin{proof}
The case $k=1$ is obvious. We assume that $k \geq 2$ and $\lambda_1=\lambda_{\max}$. we begin my recalling the following from the proof of Lemma 3.6 in \cite{Dieter05Qk}
\begin{equation*}\label{eq:Pre Q_k derivative}
\frac{n}{k(n-k+1)}\frac{S^2_{k-1;i}}{S^2_{k-1}} \leq \frac{\bd Q_k}{\bd \lambda_i}=\frac{S^2_{k-1;i}-S_{k;i}S_{k-2;i}}{S_{k-1}^2} \leq \frac{S^2_{k-1;i}}{S^2_{k-1}}. \tag{2.4}
\end{equation*}
Hence, we have the right hand side inequality of (\ref{eq:Pre Uniform Parabolicity}). For the left hand side inequality, we observe that $\lambda_i \leq \lambda_1$ implies $\lambda_i S_{k-1;i} \leq \lambda_1 S_{k-1;1} $. Therefore, we obtain
\begin{equation*}
n \lambda_1 S_{k-1;1} \geq  \sum^n_{i=1}\lambda_i S_{k-1;i}=\frac{ n \binom{n-1}{k-1}}{\binom{n}{k}}  S_k =k S_k. 
\end{equation*}
Also, $S_{k-1;i} \geq S_{k-1;1}$ and combining the above  yields 
\begin{equation*}
\frac{\bd Q_k}{\bd \lambda_i} \geq \frac{n}{k(n-k+1)}\frac{S^2_{k-1;1}}{S^2_{k-1}} \geq \frac{1}{k\lambda_1^2}\frac{\lambda_1^2 S^2_{k-1;1}}{S^2_{k-1}}\geq \frac{k}{n^2\lambda_1^2}\frac{S^2_{k}}{S^2_{k-1}} 
\geq \frac{1}{n^2}\frac{Q_k^2}{\lambda_1^2}. 
\end{equation*}
Next, to show (\ref{eq:Pre Parabolic Coefficients}), we employ (\ref{eq:Pre Q_k derivative}) again to obtain
\begin{equation*}
\frac{\bd Q_k}{\bd \lambda_i}  \leq \frac{S_{k-1;i}^2}{S_{k-1}^2} = \frac{1}{\lambda_i^2}\frac{\lambda_i^2 S_{k-1;i}^2}{S_{k-1}^2}\leq \frac{1}{\lambda_i^2}\frac{S_k^2}{S_{k-1}^2}=\frac{Q_k^2}{\lambda_i^2}.
\end{equation*}
The left hand side of  ($\ref{eq:Pre Reaction Inequality}$)  is proven in Lemma 3.7 in \cite{Dieter05Qk}. The right hand side inequality readily follows  by (\ref{eq:Pre Parabolic Coefficients}),
since  $\displaystyle \Ak \coloneq \sum_{i=1}^n \lambda^2_i D_i Q_k  \leq n Q_k^2 $.
\end{proof}

\begin{proposition} Assume $\Omega_0$ and $\Sigma_0$ satisfy the assumptions in \emph{Theorem \ref{thm:INT Existence} }. Let $\Sigma_t$ be a convex complete smooth graph solution of \eqref{eq:INT Qkflow} with the positive $Q_k$ curvature. Then, the following hold
\begin{align*}
&\bd_t \psi =\cL \, \psi \tag{2.5}\label{eq:Pre psi_t}\\
&\bd_t  g_{ij} = -2 Q_k h_{ij}  \tag{2.6}\label{eq:Pre g_t}\\
&\bd_t  g^{ij} = 2 Q_k h^{ij}  \tag{2.7}\label{eq:Pre 1/g_t}\\
&\bd_t  \vec{n} = -(\nabla_j Q_k) F^{j}  \tag{2.8}\label{eq:Pre n_t}\\
&\bd_t h_{ij} = \cL \, h_{ij}+\frac{\bd^2 Q_k}{\bd h_{pq}\bd h_{rs}}\nabla_i h_{pq}\nabla_j h_{rs}-2Q_k h_{il}h^l_j+\Ak\,  h_{ij} \tag{2.9}\label{eq:Pre h_t}\\
&\bd_t Q_k  = \cL\, Q_k +\Ak Q_k   \tag{2.10}\label{eq:Pre Qk_t}\\
&\bd_t \grad^2  =  \cL \, \grad^2 - 6\|\nabla \grad \|^2_\cL- 2\Ak \, \grad^2   \tag{2.11}\label{eq:Pre nu_t}\\
\end{align*}
\end{proposition}

\begin{proof}
\eqref{eq:Pre g_t} - \eqref{eq:Pre Qk_t} are given in \cite{A94NonlinearFlow} (see also in  \cite{Dieter05Qk}). Equation \eqref{eq:Pre psi_t} readily follows from the definition 
$\psi \coloneq (M-\bar u)_+$,
where $\bar u \coloneq  \langle  F , \vec{e}_{n+1} \rangle$ and 
\begin{align*}
\cL \, \bar u = \cL\, \langle  F , \vec{e}_{n+1} \rangle =  \langle \frac{\bd Q_k}{\bd h_{ij}}\,  \nabla_i \nabla_j F , \vec{e}_{n+1} \rangle = \langle \frac{\bd Q_k}{\bd h_{ij}}\,  h_{ij}\vec{n} , \vec{e}_{n+1} \rangle =\langle Q_k \,  \vec{n} , \vec{e}_{n+1} \rangle = \langle \bd_t F ,\vec{e}_{n+1} \rangle = \bd_t \bar u.
\end{align*}

To show \eqref{eq:Pre nu_t},  we derive from $\grad \coloneq \langle  \vec{n} ,\vec{e}_{n+1}\rangle^{-1}$ that $\nabla_i \grad = -\langle \nabla_i \vec{n}, \vec{e}_{n+1} \rangle\, \grad^2= \langle h_{ij}F^j ,\vec{e}_{n+1} \rangle \,\grad^2$. Hence, 
\begin{align*}
\cL \, \grad = &  \frac{\bd Q_k}{\bd h_{ij}} \nabla_i \nabla_j \grad 
  =  \frac{\bd Q_k}{\bd h_{ij}} \nabla_i \big( \langle h_{jm}F^m, \vec{e}_{n+1} \rangle \, \grad^2\big )\\
  = &   \langle \big( \frac{\bd Q_k}{\bd h_{ij}} \nabla_i h_{jm} \big ) F^m, \vec{e}_{n+1} \rangle \, \grad^2 + \langle \frac{\bd Q_k}{\bd h_{ij}} h^m_i h_{jm} \, \vec{n},\vec{e}_{n+1} \rangle \,\grad^2 + 2 \frac{\bd Q_k}{\bd h_{ij}} \langle h_{jm}F^m,\vec{e}_{n+1} \rangle \langle h_{il}F^l, \vec{e}_{n+1} \rangle \, \grad^3  \\
  = &\langle (\nabla_m Q_k ) F^m, \vec{e}_{n+1} \rangle \, \grad^2 +  \Ak  \grad + 2\grad^{-1}\|\nabla \grad\|^2_{\cL}. 
\end{align*}
On the other hand, \eqref{eq:Pre n_t} gives $\bd_t \grad = \langle (\nabla_j Q_k ) F^j, \vec{e}_{n+1} \rangle \, \grad^2$. Therefore,
\begin{align*}
\bd_t \grad^2 &= 2\grad \bd_t \grad = 2\grad( \langle (\nabla_m Q_k ) F^m, \vec{e}_{n+1} \rangle \, \grad^2)\\
&= 2\grad \cL \, \grad -4 \|\nabla \grad \|^2_{\cL} -2\Ak \grad^2 =\cL \, \grad^2 - 6\|\nabla \grad \|^2_\cL- 2\Ak \grad^2.
\end{align*}
\end{proof}

If $\psi \coloneq (M-\bar u)_+$, for a given $M>0$, then we have the following two estimates.

\begin{theorem}[Gradient estimate]\label{thm:Pre Gradient}
 Assume $\Omega_0$ and $\Sigma_0$ satisfy the assumptions in \emph{Theorem \ref{thm:INT Existence}}. Let $\Sigma_t$ be a convex complete smooth graph solution of \eqref{eq:INT Qkflow} with the positive $Q_k$ curvature defined on $M^n \times [0,T]$, for some $T>0$. Then
\begin{equation*}
\psi(p,t)\grad(p,t) \leq \sup_{p \in M^n}\psi (p,0)\grad (p,0). 
\end{equation*}
\end{theorem}

\begin{proof}
By combining (\ref{eq:Pre psi_t}) and (\ref{eq:Pre nu_t}), we have
$$\bd_t  (\grad^2 \psi^2) =\cL ( \grad^2 \psi^2) -  \langle 6 \psi\nabla \grad + 2 \grad  \nabla \psi, \nabla (\grad \psi )\rangle_\cL - 2\Ak \grad^2\psi^2.$$ 
Since the conditions (ii), (iii) in Theorem \ref{thm:INT Existence} mean that $\psi$ is compactly supported, it follows by the  maximum principle that  
$\displaystyle \sup_{p \in M^n}\grad(p,t)\psi(p,t) \leq \sup_{p \in M^n}\grad(p,0)\psi(p,0) $, which yields the desired result.
\end{proof}

\begin{theorem}[Lower bound of speed]\label{thm:Pre Lower Bound of Q_k}
 Assume $\Omega_0$ and $\Sigma_0$ satisfy the assumptions in \emph{Theorem \ref{thm:INT Existence}}. Let $\Sigma_t$ be a convex complete smooth graph solution of \eqref{eq:INT Qkflow} with the positive $Q_k$ curvature defined on $M^n \times [0,T]$,  for some $T>0$. Then, 
\begin{equation*}
\psi(p,t)^{-1}  Q_k(p,t) \geq  \inf_{p \in M^n}\psi(p,0)^{-1} Q_k(p,0).
\end{equation*}
\end{theorem}

\begin{proof}
From  \eqref{eq:Pre psi_t}, \eqref{eq:Pre Qk_t}, we derive 
$$ \bd_t (\psi Q_k^{-1}) = \cL (\psi Q_k^{-1})- 2Q_k\langle \nabla (\psi Q_k^{-1}) ,\nabla Q_k^{-1} \rangle_{\cL}-\Ak Q_k^{-1}\psi.$$ Thus,  the maximum principle gives the desired result.
\end{proof}

The following result will be used in the last section. 

\begin{corollary}[Lower bound of $Q_k$ under the Mean curvature flow]\label{cor:Pre Lower Bound of Q_k}
 Assume $\Omega_0$ and $\Sigma_0$ satisfy the assumptions in \emph{Theorem \ref{thm:INT Existence}}. Let $\Sigma_t$ be a strictly convex complete smooth graph solution of the mean curvature flow \emph{($*^n_1$)} defined on $M^n \times [0,T]$ for some $T>0$. Then, 
\begin{equation*}
\psi(p,t)^{-1}  Q_k(p,t) \geq  \inf_{p \in M^n}\psi(p,0)^{-1} Q_k(p,0). 
\end{equation*}
\end{corollary}

\begin{proof}
Since $Q_k$ is a convex function and the mean curvature $H\coloneq Q_1$ is a linear function of $\lambda$, we have $\bd_t Q_k \geq \Delta Q_k$ under the MCF. Hence,  the desired result follows in the same manner as  Theorem \ref{thm:Pre Lower Bound of Q_k}.
\end{proof}

\section{Speed estimate} In this section we will obtain a local upper bound on the speed $Q_k$. We will use the gradient function $\upsilon$ to
localize our estimate in the spirit of the well known  Caffarelli, Nirenberg, and Spruck estimate  in \cite{CNS88trick}. A similar technique was used by
Ecker and Huisken  \cite{HE91MCFexist} in the context of the Mean curvature flow  to obtain a  local bound  on  $|A|^2$. Our proof is similar to that
in \cite{HE91MCFexist}.

\begin{theorem}[Speed estimate]\label{thm:SE Speed Estimate}
 Assume $\Omega_0$ and $\Sigma_0$ satisfy the assumptions in \emph{Theorem \ref{thm:INT Existence}}. Let $\Sigma_t$ be a convex complete smooth graph solution of \eqref{eq:INT Qkflow} with the positive $Q_k$ curvature defined on $M^n\times [0,T)$. Given a constant $M$, we have 
\begin{equation*}
 (\psi Q_k)^2(p,t) \leq  \max\bigg\{ 10n^2 \sup_{Q_M} \grad^4(\cdot,t), \, 2 \sup_{Q_M} \grad^2(\cdot,t) \sup_{p\in M^n} (Q_k\psi)^2(\cdot,0)\bigg\}
\end{equation*}
where $Q_M = \{(p,s)\in M^n\times [0,t]:\bar u(p,s) \leq M\}$. 
\end{theorem}

\begin{proof}
Given a time $T_0 \in [0,T)$, we define the set  $Q_M = \{(p,s)\in M^n\times [0,T_0]:\bar u(p,s) \leq M\}$ and we will prove  that 
\begin{equation*}
 (\psi Q_k)^2(p,T_0) \leq  \max\bigg\{ 10n^2 \sup_{Q_M} \grad^4(\cdot,t), \, 2 \sup_{Q_M} \grad^2(\cdot,t) \sup_{p\in M^n} (Q_k\psi)^2(\cdot,0)\bigg\}. 
\end{equation*}
Let  $\cK \coloneq \sup_{Q_M}  \grad^2$ and define the  function $\varphi$ depending on $\grad^2$ by
\begin{equation*}
\varphi (\grad^2) = \frac{\grad^2}{2\cK - \grad^2}. 
\end{equation*}
The evolution equation of $\grad^2$ in (\ref{eq:Pre nu_t}) yields
\begin{align*}
\frac{\bd}{\bd t} \varphi(\grad^2) & = \varphi' ( \cL\, \grad^2  - 6\|\nabla \grad \|^2_\cL - 2\Ak \grad^2 ) = \cL \, \varphi -\varphi''\|\nabla \grad^2 \|^2_\cL - \varphi' ( 6\|\nabla \grad \|^2_\cL+2\Ak \grad^2) 
\end{align*}
which  combined with \eqref{eq:Pre Qk_t} yields 
\begin{align*}
\frac{\bd}{\bd t} (Q_k^2 \varphi) &= \cL(Q_k^2 \varphi) -2\langle\nabla \varphi , \nabla Q_k^2\rangle_\cL -2\varphi \|\nabla Q_k \|^2_\cL  
\\& \,\,  - (4\varphi'' \grad^2 + 6\varphi')Q_k^2 \|\nabla \grad \|^2_\cL + 2 \Ak Q_k^2\,  (\varphi - \varphi' \grad^2).  
\end{align*}
Observe the following
\begin{align*}
-2\langle\nabla \varphi , \nabla Q_k^2\rangle_\cL  = & -2Q_k\langle\nabla \varphi , \nabla Q_k\rangle_\cL  + \varphi^{-1}Q_k^2 \, \|\nabla \varphi \|^2_\cL - \varphi^{-1}\langle\nabla \varphi , \nabla (\varphi Q_k^2)\rangle_\cL \\
 \leq  &\quad  2\varphi  \|\nabla Q_k \|^2_\cL  + \frac{3}{2}\varphi^{-1}Q_k^2 \|\nabla \varphi \|^2_\cL- \varphi^{-1}\langle\nabla \varphi , \nabla (\varphi Q_k^2)\rangle_\cL.
\end{align*}
Hence, the following inequality holds
\begin{align*}\label{eq:SE speed I}
\frac{\bd}{\bd t} (Q_k^2 \varphi) \leq &\,\cL  (Q_k^2 \varphi) - \varphi^{-1}\langle\nabla \varphi , \nabla ( Q_k^2 \varphi)\rangle_\cL \tag{3.1}  \\&- (4\varphi'' \grad^2 + 6\varphi' -6\varphi^{-1} \varphi'^2\grad^2)Q_k^2\|\nabla \grad\|^2_\cL + 2\Ak Q_k^2 \, (\varphi - \varphi' \grad^2). 
\end{align*}
On the other hand, a direct computation gives   the following identities  
\begin{align*}
\varphi - \varphi' \grad^2 = -\varphi^2, \quad \varphi^{-1} \nabla \varphi = 4\cK \varphi \grad^{-3} \nabla \grad , \quad  4\varphi'' \grad^2 + 6\varphi' -6\varphi^{-1} \varphi'^2\grad^2 = \frac{4\mathcal{K}}{(2\cK-\grad^2)^2}\varphi. 
\end{align*}
Setting  $f \coloneq  Q_k^2 \varphi(\grad^2)$ in \eqref{eq:SE speed I} and applying  the identities above and also  $Q_k^2\leq  n \Ak$ (see in  (\ref{eq:Pre Reaction Inequality}) )  gives
\begin{align*}
\frac{\bd}{\bd t } f \leq  \cL \,  f   -4\cK\varphi\grad^{-3} \langle\nabla \grad,\nabla f\rangle_\cL -\frac{4\cK}{(2\cK-\grad^2)^2}\|\nabla \grad\|^2_\cL \, f -2\Ak  \varphi \, f. 
\end{align*}

We will next consider the evolution of $f\psi^2$ for our given cut off function $\phi$. We have seen in  
 (\ref{eq:Pre psi_t}) that  $ \bd_t \psi^2=\cL \psi^2 - 2\|\nabla \psi \|^2_{\cL} $, on the support of $\psi$. Combining this with the evolution of $f$ yields  
\begin{align*}
\frac{\bd}{\bd t }(f\psi^2) \leq  & \cL (f \psi^2) -2\langle\nabla \psi^2, \nabla f\rangle_\cL -2f\|\nabla \psi\|^2_\cL   \\
& -4\cK\varphi\grad^{-3}\psi^2 \langle\nabla \grad,\nabla f\rangle_\cL -\frac{4\cK}{(2\cK-\grad^2)^2}f\psi^2\|\nabla \grad\|^2_\cL -\frac{2}{n}f^2\psi^2
\end{align*}
We compute the following
\begin{align*}
-4\cK\varphi\grad^{-3} \psi^2\langle\nabla \grad,\nabla f\rangle_\cL  &= -4\cK\varphi\grad^{-3} \langle\nabla \grad,\nabla (f \psi^2)\rangle_\cL +8\cK\varphi\grad^{-3} f\psi\langle\nabla \grad,\nabla  \psi\rangle_\cL \\
&\leq -4\cK\varphi\grad^{-3} \langle\nabla \grad,\nabla (f \psi^2)\rangle_\cL +\frac{4\mathcal{K}f\psi^2 \|\nabla \grad \|^2_\cL}{(1-\mathcal{K}\grad^2)^2} + 4\cK\varphi^2  \frac{(1-\mathcal{K}\grad^2)^2}{\grad^{6}}f \|\nabla \psi \|^2_\cL \\
&= -4\cK\varphi\grad^{-3} \langle\nabla \grad,\nabla (f \psi^2)\rangle_\cL +\frac{4\mathcal{K}}{(1-\mathcal{K}\grad^2)^2}f\psi^2\|\nabla \grad\|^2_\cL +  
 4\mathcal{K}\grad^{-2}f\|\nabla \psi\|^2_\cL.
\end{align*}
In addition, we know
\begin{equation*}
-2\langle\nabla \psi^2, \nabla f\rangle_\cL = -4\psi^{-1} \langle\nabla \psi, \nabla (f \psi^2) \rangle_\cL +8 f \|\nabla \psi \|^2_\cL. 
\end{equation*}
Therefore, by using the equations above, we can reduce the evolution equation of $f\psi^2$ to
\begin{align*}
\frac{\bd}{\bd t }(f\psi^2) \leq & \cL (f\psi^2)-\langle 4\psi^{-1} \nabla \psi + 4\cK\varphi\grad^{-3} \nabla \grad,\nabla (f \psi^2)\rangle_\cL  +(6+4\cK\grad^{-2})f\|\nabla \psi\|^2_\cL  -\frac{2}{n}f^2\psi^2. 
\end{align*}
Applying  the inequality $D_i Q_k \leq 1$ shown in (\ref{eq:Pre Uniform Parabolicity}) yields
\begin{align*}
\|\nabla \psi \|^2_\cL =& \|\langle \nabla F,\vec{e}_{n+1}\rangle \|^2_\cL\leq \sum^{n+1}_{m=1}\|\nabla \langle  F,\vec{e}_{m}\rangle \|^2_\cL =\sum^{n+1}_{m=1}\frac{\bd Q_k}{\bd h_{ij}}\nabla_i \langle  F,\vec{e}_m\rangle\nabla_j \langle  F,\vec{e}_m\rangle \\
= & \frac{\bd Q_k}{\bd h_{ij}}\sum^{n+1}_{m=1} \langle  F_i,\vec{e}_m\rangle \langle  F_j,\vec{e}_m\rangle =\frac{\bd Q_k}{\bd h_{ij}} \langle F_i,F_j\rangle= \frac{\bd Q_k}{\bd h_{ij}}g_{ij} = \sum^n_{i=1} D_i Q_k\leq n.
\end{align*}
Hence, by the definition of $\cK$ and $\grad \geq 1$, we have
\begin{equation*}
(6+4\cK\grad^{-2})f \|\nabla \psi \|^2_\cL \leq 10n \cK \grad^{-2} f \leq 10 n \cK f
\end{equation*}
Combining the above inequalities we finally obtain 
\begin{align*}
\frac{\bd}{\bd t }(f\psi^2) \leq & \cL (f\psi^2)-\langle 4\psi^{-1} \nabla \psi + 4\cK\varphi\grad^{-3} \nabla \grad,\nabla (f \psi^2)\rangle_\cL  +10 n \cK f -\frac{2}{n}f^2\psi^2.  
\end{align*}
Since $\psi$ is compactly supported by the conditions (ii), (iii) in Theorem \ref{thm:INT Existence}, $f\psi^2$ attains its maximum $\cM$ on $M^n \times [0,T_0] $ at some $(p_0,t_0)$. If $t_0 >0$, then at $(p_0,t_0)$,  we obtain 
\begin{align*}
\frac{2}{n}\cM f =\frac{2}{n}f^2 \psi^2 \leq 10 n \cK f. 
\end{align*}
Since  $ f\psi^2 =\varphi(\grad^2)Q_k^2\psi^2\leq  (Q_k\psi)^2$, the following holds
\begin{align*}
\cM \leq \max\bigg\{5n^2 \cK,\sup_{p \in M^n}f\psi^2(\cdot,0) \bigg\}\leq \max\bigg\{5 n^2\cK ,\sup_{p \in M}(Q_k\psi)^2(\cdot,0)\bigg\}. 
\end{align*}
Finally, $  (Q_k\psi)^2  \leq \grad^2 Q_k^2\psi^2 \leq 2\cK \varphi(\grad^2) Q_k^2\psi^2 =2\cK f\psi^2$  and $f\psi^2(p,T_0) \leq \cM$ imply 
\begin{equation*}
 (Q_k\psi)^2(p,T_0) \leq  \max\bigg\{ 10n^2 \sup_{Q_M} \grad^4(\cdot,t),2 \sup_{Q_M} \grad^2(\cdot,t) \sup_{p\in M^n} (Q_k\psi)^2(\cdot,0)\bigg\}. 
\end{equation*}
\end{proof}

\section{Curvature estimate}
In this section we will derive a local   upper bound  on  the largest principal curvature ${\lambda}_{\max}$ of
 $M_t$. We will employ  a Pogorelov type computation  with respect to $h_{ii}$ using 
a technique that was introduced by Sheng, Urbas, and Wang in \cite{SUX04Qk} for the elliptic setting. 
The following known formula and will be used in the proof. 

\begin{proposition}\label{prop:CE Euler's formula}
\emph{(The Euler's formula)} Let $\Sigma$ be a smooth hypersurface, and $F:M^n \rightarrow \rno$ be a smooth immersion with $F(M^n)=\Sigma$. Then, for all $p \in M^n$ and $i \in \{1,\cdots,n\}$, the following holds
\begin{align*}
\frac{h_{ii}(p)}{g_{ii}(p)} \leq \lambda_{\max}(p). 
\end{align*}
\end{proposition}

\begin{proof}
Assume $\{ E_1(p), \cdots,E_n(p) \}$ is an orthonormal basis of $T\Sigma_{F(p)}$ satisfying $
L(E_j(p)) =\lambda_j(p)E_j(p)$, where $L$ is the Weingarten map. Let $\nabla_i F \coloneq F_i = a_{ij}E_j$. Then, $\displaystyle g_{ii}=\sum_{j=1}^n (a_{ij})^2$. Thus,
\begin{equation*}
h_{ii}=\langle L(F_i),F_i\rangle =\sum_{ j=1}^n(a_{ij})^2\lambda_j \leq \sum_{ j=1}^n(a_{ij})^2\lambda_{\max} = g_{ii}\lambda_{\max}. 
\end{equation*}
\end{proof}

\begin{theorem}[Curvature estimate]\label{thm:CE Local Pinching Estimate}
 Assume $\Omega_0$ and $\Sigma_0$ satisfy the assumptions in \emph{Theorem \ref{thm:INT Existence}}. Let $\Sigma_t$ be a convex complete smooth graph solution of \eqref{eq:INT Qkflow} with the positive $Q_k$ curvature defined on $M^n\times [0,T)$. Then, for any given a constant $M$, we have 
\begin{equation*}
 (\psi^2\lambda_{\max}) (p,t) \leq \exp(2nt\, \sup_{Q_M}Q_k^2)\max \bigg\{ 5M, \sup_{p \in M^n} (\psi^2\lambda_{\max})(p,0) \bigg\}
\end{equation*}
where $Q_M =\{(p,s)\in M^n \times [0,t]:\bar u (p,s)\leq M\}$.
\end{theorem}

\begin{proof}
Given  $T_0 \in [0,T)$, we define $Q_M = \{(p,s)\in M^n\times [0,T_0]:\bar u(p,s) \leq M\}$ and we  will prove that 
\begin{equation*}
 (\psi^2\lambda_{\max}) (p,T_0) \leq \exp(2nT_0\, \sup_{Q_M}Q_k^2)\max \bigg\{ 5M, \sup_{p \in M^n} (\psi^2\lambda_{\max})(p,0) \bigg\}. 
\end{equation*}
We set  $\cA=\sup_{Q_M}Q_k^2$. By the conditions (ii), (iii) in Theorem \eqref{eq:INT Qkflow}, $\exp(-2nt\cA )\psi^2\lambda_{\max}$ attains its maximum  in $M^n \times [0,T_0]$ at some point $(p_0,t_0)$.  If $t_0 =0$, we obtain the desired result. So, we may assume $t_0 >0$. First we choose a chart $(U,\varphi)$ with $p_0 \in \varphi(U) \subset M^n$ such that the covariant derivatives $\{\nabla_iF(p_0,t_0):i=1,\cdot,n\}$ form an orthonormal basis of $(T\Sigma_{t_0})$ satisfying
\begin{align*}
g_{ij}(p_0,t_0)=\delta_{ij}, && h_{ij}(p_0,t_0)=\delta_{ij}\lambda_i(p_0,t_0), && \lambda_1(p_0,t_0) =\lambda_{\max}(p_0,t_0). 
\end{align*}
Then,  $h_{11}(p_0,t_0) = \lambda_{\max}(p_0,t_0) $, $g_{11}(p_0,t_0) = 1$ hold. Next, we define the  function $w: U \times [0,T_0] \rightarrow \mathbb{R}$ by
\begin{align*}
w \coloneq \exp(-2nt \, \cA )\psi^2 \frac{ h_{11}}{g_{11}}. 
\end{align*}
Notice that if $t \neq t_0$, the covariant derivatives $\{\nabla_i F(p_0,t)\}_{i =1, \cdots, n}$ may fail to form an orthonormal basis of $(T\Sigma)_{F(p_0,t)}$. However, Proposition \ref{prop:CE Euler's formula} applies  for every chart and immersion. So, for all points 
$(p,t)\in \varphi(U)\times [0,T_0]$, we have 
\begin{equation*}
w(p,t) \leq \exp(-2nt\cA )\psi^2 \lambda_{\max} (p,t) \leq \exp(-2nt_0\cA )\psi^2 \lambda_{\max} (p_0,t_0) = w(p_0,t_0) 
\end{equation*}
implying  that  $w$ attains its maximum at $(p_0,t_0)$.  Since $\nabla g_{11}=0$, the following holds on the support of $\psi$
\begin{align*}\label{eq:CE Pinching Gradeint}
\frac{\nabla_i w}{w} =  2\frac{\nabla_i \psi}{\psi}+ \frac{\nabla_i h_{11}}{h_{11}}. \tag{4.1}
\end{align*}
Differentiating  the equation above we obtain 
\begin{align*}
\frac{\nabla_i \nabla_j w}{w}-\frac{\nabla_i  w \nabla_j w}{w^2} = 2\frac{\nabla_i\nabla_j \psi}{\psi}-  2\frac{\nabla_i \psi \nabla_j \psi}{\psi^2}+\frac{\nabla_i\nabla_j h_{11}}{h_{11}}-\frac{\nabla_i h_{11}\nabla_j h_{11}}{(h_{11})^2}. 
\end{align*}
Multipling  by ${\displaystyle \frac{\bd Q_k}{\bd h_{ij}}}$ and summing over   all $i,j$ yields 
\begin{align*}
\frac{\cL\, w}{w}-\frac{\|\nabla  w \|^2_{\cL}}{w^2} =  2\frac{\cL\, \psi}{\psi}- 2\frac{\|\nabla \psi \|^2_{\cL}}{\psi^2}+\frac{\cL\, h_{11}}{h_{11}}-\frac{\|\nabla  h_{11}\|^2_{\cL}}{(h_{11})^2}. 
\end{align*}
On the other hand, on the support of $\psi$, the following holds
\begin{align*}
\frac{\bd_t w}{w} = -2n\cA + 2\frac{\bd_t \psi}{\psi}+ \frac{\bd_t h_{11}}{h_{11}}-\frac{\bd_t g_{11}}{g_{11}}. 
\end{align*}
Recall that $\bd_t \psi  = \cL \psi$  by , by  \eqref{eq:Pre psi_t}, $ \bd_t g_{11} = -2Q_k h_{11}$ by   \eqref{eq:Pre g_t}, and 
also that 
\begin{align*}
\bd_t h_{11} =\cL\, h_{11} +\frac{\bd^2 Q_k}{\bd h_{ij} \bd h_{ml}}\nabla_1 h_{ij}\nabla_1h_{ml}-2Q_k h_{1i}h^i_1 +\Ak h_{11} 
\end{align*}
by  in \eqref{eq:Pre h_t}. 
Combining the equations above yields
\begin{align*}\label{eq:CE Pinching 1st}
\frac{\cL\, w}{w}-\frac{\|\nabla  w \|^2_{\cL}}{w^2} -\frac{\bd_t w}{w} =  & - 2\frac{\|\nabla \psi \|^2_{\cL}}{\psi^2}-\frac{\|\nabla  h_{11}\|^2_{\cL}}{(h_{11})^2}-\frac{1}{h_{11}}\frac{\bd^2 Q_k}{\bd h_{ij} \bd h_{ml}}\nabla_1 h_{ij}\nabla_1h_{ml} \tag{4.2} \\
& + 2n\cA+\frac{2Q_k h_{1i}h^i_1}{h_{11}} - \Ak - 2Q_k \frac{h_{11}}{g_{11}}. 
\end{align*}
Observe next that   $\Ak \leq n Q_k^2 \leq n\cA$ by  \eqref{eq:Pre Reaction Inequality}. Also, at   $(p_0,t_0)$, $\displaystyle \frac{2Q_k h_{1i}h^i_1}{h_{11}} - 2Q_k \frac{h_{11}}{g_{11}}=0$ holds. Moreover,  Proposition \ref{prop:Pre 3rd Order} implies that 
\begin{align*}
-\frac{1}{h_{11}}\frac{\bd^2 Q_k}{\bd h_{ij} \bd h_{ml}}\nabla_1 h_{ij}\nabla_1h_{ml} & \geq 2\sum^n_{i=2} -\frac{D_1Q_k-D_iQ_k}{\lambda_1(\lambda_1-\lambda_i)} |\nabla_1 h_{1i}|^2 
\end{align*}
holds at the point  $(p_0,t_0)$. Furthermore, by the definition of the operator $\cL$, at the point $(p_0,t_0)$ we have 
\begin{align*}
\frac{\|\nabla w \|^2_{\cL}}{w^2} \geq 0, \qquad  \frac{\|\nabla \psi \|^2_{\cL}}{\psi^2} = \sum^n_{i=1}\frac{\bd Q_k}{\bd \lambda_i} \frac{|\nabla_i \psi|^2}{\psi^2}, \qquad  \frac{\|\nabla  h_{11}\|^2_{\cL}}{(h_{11})^2}=\sum^n_{i=1}\frac{\bd Q_k}{\bd \lambda_i} \frac{|\nabla_i h_{11}|^2}{\lambda_1^2}. 
\end{align*}
We conclude from (\ref{eq:CE Pinching 1st}) that at  the maximum point  $(p_0,t_0)$ of $w$,  
the following holds 
\begin{align*}\label{eq:CE Pinching 2nd}
 n\cA \leq  2\sum^n_{i=1}\frac{\bd Q_k}{\bd \lambda_i} \frac{|\nabla_i \psi|^2}{\psi^2}+\sum^n_{i=1}\frac{\bd Q_k}{\bd \lambda_i} \frac{|\nabla_i h_{11}|^2}{\lambda_1^2}+\sum^n_{i=2}  \frac{2(D_1Q_k-D_iQ_k)}{\lambda_1(\lambda_1-\lambda_i)} |\nabla_i h_{11}|^2\tag{4.3}
\end{align*}

Next, we define the following sets
\begin{align*}
I=\{i\in (1,\cdots,n) :\frac{\bd Q_k}{\bd \lambda_i} < 4\frac{\bd Q_k}{\bd \lambda_1} \} \quad \mbox{and} \quad 
J=\{j\in (1,\cdots,n) : \frac{\bd Q_k}{\bd \lambda_j}\geq 4\frac{\bd Q_k}{\bd \lambda_1} \}. 
\end{align*}
Since $w$ attains its maximum at $(p_0,t_0)$, $\nabla w(p_0,t_0) =0$ holds. Thus,  
by  \eqref{eq:CE Pinching Gradeint}  
\begin{align*}
2\sum^n_{i=1} \frac{\bd Q_k}{\bd \lambda_i}\frac{|\nabla_i \psi|^2}{\psi^2} &= 2\sum_{i \in I} \frac{\bd Q_k}{\bd \lambda_i}\frac{|\nabla_i \psi|^2}{\psi^2} + 2\sum_{j \in J} \frac{\bd Q_k}{\bd \lambda_j}\frac{|\nabla_j \psi|^2}{\psi^2} = 2\sum_{i \in I} \frac{\bd Q_k}{\bd \lambda_i}\frac{|\nabla_i \psi|^2}{\psi^2} + \frac{1}{2}\sum_{j \in J} \frac{\bd Q_k}{\bd \lambda_j}\frac{|\nabla_j h_{11}|^2}{h_{11}^2} 
\end{align*}
and
\begin{align*} 
\sum^n_{i=1} \frac{\bd Q_k}{\bd \lambda_i}\frac{|\nabla_i h_{11}|^2}{(h_{11})^2} &= \sum_{i \in I} \frac{\bd Q_k}{\bd \lambda_i}\frac{|\nabla_i h_{11}|^2}{(h_{11})^2} + \sum_{j \in J} \frac{\bd Q_k}{\bd \lambda_j}\frac{|\nabla_j h_{11}|^2}{(h_{11})^2} = 4\sum_{i \in I} \frac{\bd Q_k}{\bd \lambda_i}\frac{|\nabla_i \psi|^2}{\psi^2} + \sum_{j \in J} \frac{\bd Q_k}{\bd \lambda_j}\frac{|\nabla_j h_{11}|^2}{h_{11}^2}
\end{align*}
and by  adding  the two equations above we obtain
\begin{equation*}
2\sum^n_{i=1} \frac{\bd Q_k}{\bd \lambda_i}\frac{|\nabla_i \psi|^2}{\psi^2}+\sum^n_{i=1}\frac{\bd Q_k}{\bd \lambda_i}\frac{|\nabla_i h_{11}|^2}{\lambda_1^2}=6\sum_{i \in I} \frac{\bd Q_k}{\bd \lambda_i}\frac{|\nabla_i \psi|^2}{\psi^2}+\frac{3}{2}\sum_{j \in J} \frac{\bd Q_k}{\bd \lambda_j}\frac{|\nabla_j h_{11}|^2}{\lambda_1^2}. 
\end{equation*} 
However, we know $1 \not\in J$ and $\lambda_1 \neq \lambda_j$. Hence, for $j \in J $, the definition of $J$ leads to
\begin{equation*}
\frac{\bd Q_k}{\bd \lambda_1} -\frac{\bd Q_k}{\bd \lambda_j} \leq - \frac{3}{4}\frac{\bd Q_k}{\bd \lambda_j}. 
\end{equation*}
On the other hand  $\lambda_1 (\lambda_1 -\lambda_j) \leq  (\lambda_1)^2  $ holds. Hence,  for $j \in J$, we obtain
\begin{equation*}
 \frac{2(D_1Q_k-D_jQ_k)}{\lambda_1(\lambda_1-\lambda_j)}|\nabla_j h_{11}|^2  \leq -\frac{3}{2} \frac{\bd Q_k}{\bd \lambda_j}\frac{|\nabla_j h_{11}|^2}{\lambda_1^2}.  
\end{equation*}
Thus, (\ref{eq:CE Pinching 2nd}) can be reduced to
\begin{align*}
 n\cA  \leq 6\sum_{i \in I} \frac{\bd Q_k}{\bd \lambda_i}\frac{|\nabla_i \psi|^2}{\psi^2}. 
\end{align*}
Applying $|\nabla_i \psi|^2 =|\langle F_i ,\vec{e}_{n+1} \rangle|^2 \leq | F_i|^2  = g_{ii} =  1$ and  the definition of $I$, we obtain
\begin{align*}
 n\cA  \leq 6\sum_{i \in I} \frac{\bd Q_k}{\bd \lambda_i}\frac{|\nabla_i \psi|^2}{\psi^2}  \leq 24\sum_{i \in I} \frac{\bd Q_k}{\bd \lambda_1}\frac{1}{\psi^2} \leq 24 n\,  \frac{\bd Q_k}{\bd \lambda_1} \psi^{-2}. 
\end{align*}
Using that  $D_1 \lambda \leq  Q_k^2\lambda_1^{-2}= Q_k^2\lambda_{\max}^{-2}$ by \eqref{eq:Pre Parabolic Coefficients}, 
$\psi \leq M$, and $Q_k^2 \leq \cA$, in the inequality above yields  
\begin{align*}
\cA \,  \psi^4  \lambda_{\max}^2\leq  24  \, Q_k^2 \psi^2 \leq 24 \,  \cA M^2
\end{align*}
implying that  $\psi^2 \lambda_{\max} (p_0,t_0)  \leq 5M$ holds. In conclusion,
\begin{align*}
w(p,t)\leq w(p_0,t_0) \coloneq \exp (-2n t_0\,\cA)\psi^2\lambda_{\max} (p_0,t_0) \leq \psi^2 \lambda_{\max} (p_0,t_0)  \leq 5M
\end{align*}
which finishes the proof of our estimate. 
\end{proof}

\section{Curvature derivative estimates}\label{sec-higher}
In this section, we will derive  local curvature derivative estimates of any order, by 
 adopting   Shi's local derivative estimates  in the  fully-nonlinear setting. 
 
\begin{proposition}\label{prop:DE Basic estimate}
Let $\Sigma_t$ be a convex smooth solution of \emph{\eqref{eq:INT Qkflow}} defined for $ t \in [0,T)$ with the positive $Q_k$ curvature. Given a set $Q \subset \rno \times [0,T)$ and for each order  $m \geq 0$, there exist  positive constants $\displaystyle\gamma_m \coloneq \gamma_m(k,m,n,\inf_{Q}Q_k,\sup_Q | A|^2)$ and $\displaystyle C_m\eqqcolon C_m(k,m,n,\inf_{Q}Q_k,\sup_Q | A|^2,\max_{ j < m}\sup_Q |\nabla^j A|^2)$ for which 
\begin{align*}\label{eq:DE Basic estimate}
\bd_t |\nabla^m A|^2 \leq \cL\, |\nabla^m A|^2 -\gamma_m  |\nabla^{m+1} A |^2 +C_m\big(|\nabla^m A|^4 +1\big)\tag{5.1}
\end{align*}
holds on $Q$. 
\end{proposition}

\begin{proof} 
Since the derivatives of $Q_k(\lambda)$ are controlled by $|A|^2$ and $(S_{k-1})^{-1} \geq C(n,k) (Q_k)^{-(k-1)}$, 
the estimate
$$(\bd_t -\cL) |A|^2 \leq  -2\gamma_0  |\nabla A |^2 +C_{0}$$
readily follows   by  \eqref{eq:Pre h_t}. 

Let us ow show the estimate for $m=1$. 
For convenience, we define $a_{ij}(A)$, $b_{ij}(A,\nabla A)$, and $c_{ij}(A)$ by
\begin{align*}
a_{ij}=\frac{\bd Q_k}{\bd h_{ij}}, \qquad  b_{ij}=\frac{\bd^2 Q_k}{\bd h_{pq}\bd h_{rs}}\nabla_i h_{pq}\nabla_j h_{rs}, \qquad
 c_{ij}=-2Q_k h_{il}h^l_j+\Ak h_{ij}
\end{align*}
Then, $\bd_t \nabla_k h_{ij} = \nabla_k \big( a_{pq}\nabla_p \nabla_q h_{ij} \big)+\nabla_k b_{ij} +\nabla_k c_{ij}$ gives
\begin{align*}
\bd_t |\nabla A|^2=&  \bd_t \big( g^{pi}g^{ql}g^{rk}\nabla_r h_{pq} \nabla_k h_{ij} \big)=2 (\bd_t \nabla_k h_{ij})\nabla^kh^{ij}+6Q_k h^{kr}\nabla_r h^{ij} \nabla_k h_{ij}\\
\leq &\, 2a_{pq} \nabla_k \nabla_p \nabla_q h_{ij}\nabla^kh^{ij} +2 \nabla_k  a_{pq} \nabla_p \nabla_q h_{ij}\nabla^k h^{ij}+C|\nabla A|^2+C\big(|\nabla^k h^{ij}\nabla_k b_{ij}|+|\nabla^k h^{ij} \nabla_k c_{ij}|\big)\\
\leq &\, 2a_{pq} \nabla_k \nabla_p \nabla_q h_{ij}\nabla^kh^{ij} +C(n,k,\sup |A|^2,\inf Q_k)\big(|\nabla^2 A| |\nabla A|^2 +|\nabla A|^4+|\nabla A|^2\big). 
\end{align*}
Commutating derivatives yields $|\nabla_k \nabla_p \nabla_q h_{ij}-\nabla_p \nabla_q \nabla_k h_{ij}|\leq C|\nabla A| $, which leads to
\begin{align*}
\bd_t |\nabla A|^2 \leq &\,  2\big(\cL\, \nabla_k h_{ij})\nabla^k h^{ij} +C\big(|\nabla^2 A| |\nabla A|^2 +|\nabla A|^4+|\nabla A|^2\big) \\
\leq &\,  \cL\, |\nabla A|^2 - 2a_{mn}\nabla_m \nabla_k h_{ij} \nabla_n \nabla^k h^{ij} +C\big(|\nabla^2 A| |\nabla A|^2 +|\nabla A|^4+|\nabla A|^2\big) \\
\leq & \, \cL\, |\nabla A|^2 - 2\gamma_1(n,k,\sup |A|^2,\inf Q_k) |\nabla^2 A|^2+C\big(|\nabla^2 A| |\nabla A|^2 +|\nabla A|^4+|\nabla A|^2\big) \\
\leq & \, \cL\, |\nabla A|^2 - \gamma_1|\nabla^2 A|^2+C\big(|\nabla A|^4+|\nabla A|^2\big).
\end{align*}
Thus, $|\nabla A|^4+|\nabla A|^2 \leq C\big(|\nabla A|^4+1\big)$ implies the desired result.

Let us consider the other cases $m \geq 2$. For an index set $\alpha\eqqcolon(\alpha_1,\cdots,\alpha_m)$ of degree $m$ and a tensor $T$, we denote by $\nabla_\alpha T$ and $\nabla^\alpha T$ the $m$-th order derivatives $\nabla_{\alpha_1} \cdots\nabla_{\alpha_m} T$ and $\nabla^{\alpha_1} \cdots\nabla^{\alpha_m} T$, respectively. Then,
\begin{align*}
\bd_t |\nabla^m A|^2=  \bd_t \Big( \sum_{|\alpha|=m} \big( \nabla^\alpha h^{ij} \nabla_\alpha h_{ij} \big) \Big) \leq 2\sum_{|\alpha|=m} \nabla^\alpha h^{ij}\,(\bd_t \nabla_\alpha h_{ij})+C |\nabla^m A|^2. 
\end{align*}
By using $|\alpha|=m$, on an orthonormal frame, we can observe the following
\begin{align*}
\big| \nabla_\alpha  c_{ij}(A) \big| \leq & \, C \big(|\nabla^m A|+|\nabla^{m-1}A||\nabla A|+|\nabla^{m-2}A|(|\nabla^2 A|+|\nabla A|^2)+\cdots+|\nabla A|^m \big) \\
\leq &\, C\big(|\nabla^m A|+\max_{j< m}\sup |\nabla^{j}A|^{m}+1) \leq C(|\nabla^m A|+1\big).
\end{align*}
In addition, since $\nabla_\alpha b_{ij}(A ,\nabla A)$ includes $(m+2)$ derivatives, we have
\begin{align*}
\big| \nabla_\alpha b_{ij}(A ,\nabla A) \big| \leq C \big(|\nabla^{m+1} A||\nabla A|+|\nabla^m A||\nabla^2 A|+|\nabla A|^m|\nabla A|^2+\max_{j< m}\sup |\nabla^{j}A|^{m+2}+1\big). 
\end{align*}
Notice that $|\nabla^2 A|$ can be $|\nabla^m A|$ in the case $m=2$. Thus, 
\begin{align*}
\big| \nabla_\alpha b_{ij}(A ,\nabla A) \big| \leq C \big(|\nabla^{m+1} A|+|\nabla^m A|^2+1\big).
\end{align*}
Moreover, $\nabla_\alpha (a_{pq}\nabla_p\nabla_q h_{ij})-a_{pq}\nabla_\alpha \nabla_p \nabla_q h_{ij} $ can be estimated as $\big| \nabla_\alpha b_{ij}(A ,\nabla A) \big|$. Thus, 
\begin{align*}
|\nabla_\alpha (a_{pq}\nabla_p\nabla_q h_{ij})-a_{pq}\nabla_\alpha \nabla_p \nabla_q h_{ij}|
 \leq  C \big(|\nabla^{m+1} A|+|\nabla^m A|^2+1\big). 
\end{align*}
By commutating derivatives, we have  
\begin{align*}
|\nabla_\alpha\nabla_p \nabla_q h_{ij}-\nabla_p\nabla_q \nabla_\alpha h_{ij}|\leq   C\big(|\nabla^m A|+\max_{j< m}\sup_Q |\nabla^{j}A|^{m}+1\big) \leq C\big(|\nabla^m A|+1\big). 
\end{align*}
Therefore,
\begin{align*}
\bd_t |\nabla^m A|^2 \leq &\,  2\sum_{|\alpha|=m} \nabla^\alpha h^{ij}\cL\, \nabla_\alpha h_{ij} +C|\nabla^m A|\big(|\nabla^{m+1} A| +|\nabla^m A|^2+|\nabla^m A|+1\big) \\
\leq & \, \cL\, |\nabla^m A|^2 - 2\gamma_m |\nabla^{m+1} A|^2+C\big(|\nabla^{m+1} A| |\nabla^m A| +|\nabla^m A|^3+|\nabla^m A|^2+|\nabla^m A|\big) \\
\leq & \, \cL \,|\nabla^m A|^2 - \gamma_m|\nabla^{m+1} A|^2+C_m\big(|\nabla^m A|^4+1\big). 
\end{align*}
\end{proof}

\begin{theorem}[Curvature derivative estimates] \label{thm:DE Higher Order Estimates}
Assume that $\Omega_0$ and $\Sigma_0$ satisfy the assumptions in  \emph{Theorem \ref{thm:INT Existence}}  and let $\Sigma_t$ be a complete convex smooth graph solution of \emph{(\ref{eq:INT Qkflow})} with the positive $Q_k$ curvature defined on $M^n \times [0,T)$. Given a constant $M>0$ and an order $m \in \mathbb{N}$, the following holds
\begin{equation*}
 \psi^2  | \nabla^m A |^2 (p,t) \leq C\big(k, m, n, M,\;   \inf_{Q_M} Q_k, \; \sup_{Q_M} |A|^2,  \; \max_{ j < m}\big(\sup_{Q_M} |\nabla^j A|^2),\; \sup_{p \in M^n} \psi^2|\nabla^m A|^2(p,0)\big)
\end{equation*}
where $Q_M \coloneqq \{ (p,s) \in M^n \times [0,t]:\bar u(p,s) \leq M\}.$
\end{theorem}

\begin{proof}
Given a time $T_0 \in [0,T)$,  we will show that 
\begin{equation*}
 \psi^2  | \nabla^m A |^2 (p,T_0) \leq C\big(k, m, n, M,\;   \inf_{Q_M} Q_k, \; \sup_{Q_M} |A|^2,  \; \max_{ j < m}\big(\sup_{Q_M} |\nabla^j A|^2),\; \sup_{p \in M^n} \psi^2|\nabla^m A|^2(p,0)\big)
\end{equation*}
holds on $Q_M \coloneqq \{ (p,s) \in M^n \times [0,T_0]:\bar u(p,s) \leq M\}.$

To this end, we define the functions  $\varphi$ and $f$ by 
\begin{align*}
\varphi  = \psi^2 \, |\nabla^m A|^2, \qquad  f  =   |\nabla^{m-1} A|^2+\theta_m+1
\end{align*}
where $\theta_m$ is a positive constant to be determined. Our estimate will follow by considering  the evolution of 
the function $\varphi^f$. In the usual setting of Shi's derivative estimates one considers the evolution of the function $f \varphi$ (c.f. in \cite{HE89MCFasymptotic, HE91MCFexist} for the result in the case of the Mean curvature flow). 

If the set $\{\varphi \geq 1 \} \coloneq  \{(p,t)\in M^n \times [0,T_0]:\varphi(p,t)\geq 1\}$ is empty, then $\psi^2 \,  | \nabla^m A |^2 (p,T_0) \leq 1$
and the desired estimate follows. So, we may assume that $\{\varphi \geq 1 \}$ is non-empty  and we will next consider the  evolution equation of
the function  $\varphi^f$ on $\{\varphi \geq 1 \}$. We have,
\begin{align*}
\bd_t \varphi^{f } = f  \varphi^{f-1} \bd_t \varphi  + (\log \varphi ) \varphi^{f } \bd_t f 
\end{align*}
and also
\begin{align*}
\cL\, \varphi^{f } = &\frac{\bd Q_k}{\bd h_{ij}}\nabla_i\big( f  \varphi^{f -1} \nabla_j \varphi  + (\log \varphi ) \varphi^{f } \nabla_j f \big) \\
= & f  \varphi^{f -1} \cL\, \varphi  + (\log \varphi ) \varphi^{f } \cL\, f  \\
&+f (f -1)\varphi^{f -2}\|\nabla \varphi  \|^2_{\cL} +|\log \varphi |^2 \varphi^{f } \|\nabla f  \|^2_{\cL} + 2\varphi^{f -1}(1+f \log \varphi )\langle \nabla \varphi ,\nabla f \rangle_{\cL}\\ 
\geq & f  \varphi^{f -1} \cL\, \varphi  + (\log \varphi ) \varphi^{f } \cL\, f  -\varphi^{f -2}\|\nabla \varphi  \|^2_{\cL} + 2\varphi^{f -1}\langle \nabla \varphi  ,\nabla f \rangle_{\cL}\\
\geq & f  \varphi^{f -1} \cL\, \varphi  + (\log \varphi ) \varphi^{f } \cL\, f  -2\varphi^{f -2}\|\nabla \varphi  \|^2_{\cL} - \varphi^{f }\|\nabla f \|_{\cL}^2.
\end{align*}
Combining the inequalities above yields 
\begin{align*}\label{eq:DE Evolution}
(\bd_t -\cL \,) \varphi^{f }  \leq  \varphi^{f}\big(f \varphi^{-1}(\bd_t - \cL\,) \varphi  + 2 \varphi^{-2} \|\nabla \varphi  \|_{\cL}^2+ (\log \varphi )(\bd_t - \cL\,) f  +  \|\nabla f  \|^2_{\cL}\big). \tag{5.2}
\end{align*}

Recall next that our cut off function $\psi$ satisfies  $\bd_t \psi^2  \leq \cL \,\psi^2$ by \eqref{eq:Pre psi_t} and combine it  with \eqref{eq:DE Basic estimate} to show
that $\varphi  \coloneq  \psi^2 \, |\nabla^m A|^2$ satisfies 
\begin{align*}
(\bd_t -\cL\,)\varphi  \leq & -2\langle \nabla \psi^2,\nabla |\nabla^m A|^2\rangle_\cL  -\gamma_m  \psi^2 |\nabla^{m+1} A |^2 +C_{m}\big(|\nabla^m A|^4\psi^2  +\psi^2\big) \\
\leq &\, C \, \psi |\nabla \psi| |\nabla^{m+1}A||\nabla^m A|-\gamma_m \,  \psi^2  |\nabla^{m+1} A |^2 +C_{m} \, |\nabla^m A|^{2} (\varphi + \psi^4\varphi^{-1}). 
\end{align*}
Using that  $\varphi \geq 1$,  $\psi \leq M$, and $|\nabla \psi|^2\leq n$ we obtain
\begin{align*}\label{eq:DE Varphi Evolution}
(\bd_t -\cL\,)\varphi  \leq  -\frac{1}{2}\gamma_m  \psi^2 |\nabla^{m+1} A |^2 +C_{m,1}|\nabla^m A|^{2} \varphi\tag{5.3}
\end{align*}
for some constant $C_{m,1}$. In addition, $\nabla \varphi  = |\nabla^m A|^2 \nabla \psi^2 + \psi^2\nabla |\nabla^m A|^2$ gives
\begin{align*}\label{eq:DE Varphi Gradient}
\|\nabla \varphi  \|^2_{\cL} \leq &\, 2\big(|\nabla^m A|^4\|\nabla \psi^2\|_{\cL}^2+\psi^4\|\nabla |\nabla^m A|^2\|_{\cL}^2\big)\leq C_{m,2} \,  \varphi  \big(|\nabla^m A|^2 + \psi^2|\nabla^{m+1} A|^2\big)\tag{5.4}
\end{align*}
for some constant $ C_{m,2}$. We set 
$$\theta_m \coloneq  4C_{m,2}/\gamma_m \quad \mbox{and} \quad  C_{m,3}= \sup_{Q_M}|\nabla^{m-1}A|^2+\theta_m+1.$$
It follows from 
\eqref{eq:DE Varphi Evolution},  \eqref{eq:DE Varphi Gradient}, $2C_{m,2}\eqqcolon \frac{1}{2}\gamma_m \theta_m  \leq \frac{1}{2}\gamma_m f$, $f \leq C_{m,3}$, and $\varphi^{-1}\leq 1$ 
that 
\begin{align*}\label{eq:DE Varphi}
f\varphi^{-1} (\bd_t - \cL) \varphi  + 2 \varphi^{-2} \|\nabla \varphi  \|_{\cL}^2   \leq   (C_{m,1}C_{m,3}+2C_{m,2}) |\nabla^m A|^2. \tag{5.5}
\end{align*}

On the other hand, \eqref{eq:DE Basic estimate} leads to
\begin{align*}\label{eq:DE f Evolution}
(\bd_t - \cL\,) f &\leq  -\gamma_{m-1} |\nabla^m A |^2 +C \\
&=  \, |\nabla^m A |^2(-\gamma_{m-1} +C\psi^2\varphi^{-1})\leq   |\nabla^m A |^2(-\gamma_{m-1} +C_{m,4}\varphi^{-1})\tag{5.6}
\end{align*}
for some constant $C_{m,4}$. Also, we have $\|\nabla f \|_{\cL}^2 \leq C|\nabla^{m-1} A|^2|\nabla^{m} A|^2$. Therefore,  for some constant $C_{m,5}$
\begin{align*}\label{eq:DE f Gradient}
\|\nabla f \|_{\cL}^2 \leq C_{m,5}|\nabla^{m} A|^2\tag{5.7}
\end{align*}
Thus, on $\{\varphi \geq 1 \}$, combining (\ref{eq:DE Varphi}), \eqref{eq:DE f Evolution}, and \eqref{eq:DE f Gradient} yields
\begin{align*}
(\bd_t -\cL \,) \varphi^{f }  \leq  \varphi^{f}|\nabla^{m} A|^2\big(C_{m,1}C_{m,3}+2C_{m,2}-\gamma_{m-1}\log \varphi +C_{m,4}\varphi^{-1}\log \varphi +C_{m,5} \big).
\end{align*}
By using $\varphi^{-1}\log\varphi \leq 1$, we can reduce the inequality above to
\begin{align*}\label{eq:DE Final}
(\bd_t -\cL \,) \varphi^{f }  \leq  \gamma_{m-1}\varphi^{f}|\nabla^{m} A|^2(C_{m,6}-\log \varphi)\tag{5.8}
\end{align*}
where $C_{m,6} \eqqcolon (C_{m,1}C_{m,3}+2C_{m,2}+C_{m,4}+C_{m,5})/\gamma_{m-1}$.

We are now ready to finish our argument. Since $\psi$ is compactly supported, $\varphi^f$ attains its maximum at some point $(p_0,t_0) \in M^n \times [0,T_0]$.
Assume $t_0 \in (0,T_0]$ and $\varphi^f (p_0,t_0) \geq 1$. Then, $\varphi(p_0,t_0) \geq 1$. Thus, at $(p_0,t_0)$, the inequality \eqref{eq:DE Final} holds. Hence,
\begin{align*}
0  \leq C_{m,6}- \log \varphi (p_0,t_0). 
\end{align*}
Since $f \geq 1$, if $\varphi(p,t)\geq 1 $ then $\varphi^f (p,t) \geq \varphi(p,t)$. Thus, 
\begin{equation*}
\varphi(p,T_0) \leq 1+\varphi^f (p,T_0)\leq 1+ \varphi^f (p_0,t_0) \leq 1+(\varphi (p_0,t_0))^{C_{m,3}} \leq 1+\exp(C_{m,6}C_{m,3})
\end{equation*}
which completes our proof. 

\end{proof}

\section{Long time existence}

In this final section, we will establish the  long time 
existence of the $Q_k$-flow (\ref{eq:INT Qkflow}), as stated in our main Theorem \ref{thm:INT Existence}.
Our proof will be based on the  a'priori estimates in sections  \ref{sec-prelim}-\ref{sec-higher}. Before we present the proof of Theorem \ref{thm:INT Existence}, we will 
introduce  some extra notation  and preliminary results.

\begin{notation}\label{not2}  We have:  

\begin{enumerate}
\item  Given a set $A \in \rno$, we denote by $\cv(V)$ its \emph{convex hull} $\{ tx+(1-t)y:  x,y \in A, t \in [0,1]\}$.
\item  Let $\Sigma$ be a convex complete (or closed) hypersurface.  If a set $V$ is a subset of  $\cv(\Sigma)$, we say $V$ is \emph{enclosed} by $\Sigma$ and use the notation  
\begin{equation*}
V \preceq \Sigma. 
\end{equation*}
In particular, if $V \bigcap \Sigma =\emptyset$ and $V \preceq \Sigma $, we use $V \prec \Sigma$.

\item $B^{n+1}_R(Y)\coloneqq \{X \in \rno: |X-Y|<R \}$ 
denotes the \emph{$(n+1)$-ball of  radius $R$ centered at $Y \in\rno$}.

\item For a convex hypsersurface $\Sigma$ and $\eta>0$, we denote by $\Sigma^\eta$ the $\eta$-envelope of $\Sigma$.
\begin{align*}
\Sigma^{\eta} = \big \{Y \in \rno :d(Y,\Sigma)=\eta,\, Y \not \in \cv(\Sigma) \big\}
\end{align*}
where $d$ is the distance function.

\item For a convex closed hypersurface $\Sigma$, we define the \textit{support function} $S:S^n\to \mathbb{R}$ by
\begin{align*}
S(v)=\max_{Y\in \Sigma}\,\langle v,Y\rangle.
\end{align*}

\item For a convex $C^2$ hypersurface $\Sigma$ and a point $X \in \Sigma$, we denote by $\lambda_{\min}(\Sigma)(X)$ the smallest principal curvature. 
Also, for any convex hypersurface $\Sigma$ and a point $X \in \Sigma$, we define 
\begin{align*}
\qquad \lambda_{\min}(\Sigma)(X) = \sup \big\{ \lambda_{\min}(\Phi)(X): \Phi \,  \text{complete (or closed) $C^2$ hypersurface}, \Sigma \preceq \Phi , X \in \Phi \big\}
\end{align*}
and also 
\begin{align*}
\lambda_{\min}^{\text{loc}}(\Sigma)(X)=\liminf_{ r \to 0} \Big \{\lambda_{\min} (\Sigma)(Y):  Y \in \Sigma \bigcap B^{n+1}_r(X)\Big\}. 
\end{align*}
\end{enumerate}
\end{notation}

\smallskip

Since we will approximate the initial hypersurface $\Sigma_0$ by its envelopes $(\Sigma_0)^\eta$, which are of class $C^{1,1}$, in order to regularized them, we introduce the convolution on the sphere.

\begin{proposition}[Convolution on $S^n$]\label{prop:LTE Convolution}
For $\epsilon \in (0,1)$, let $\varphi_\epsilon:S^n\times S^n \to \mathbb{R}$ be a smooth function satisfying
\begin{enumerate}
\item[{\em (i)}]  $\varphi_\epsilon(v,w)=\eta_\epsilon(\langle v,w\rangle)$ for a non-negative function $\eta_\epsilon:[-1,1]\to [0,+\infty)$, 
\item [{\em (ii)}] $\eta_\epsilon(r)=0$ for all $r\in [-1, 1-\epsilon]$
\item [{\em (iii)}] $\int_{S^n} \varphi_{\epsilon}(v) d s =1 $, where $ds$ denotes  the surface measure on  $S^n$  
\end{enumerate}
and  define the convolution $f*\varphi_\epsilon$ with a function $f:S^n \to \mathbb{R}$ by
\begin{align*}
f*\varphi_\epsilon (v)=\int_{S^n} f(w)\varphi_\epsilon(v,w)d s_w.  
\end{align*}
Assume that $f$ is of class $C^m(U)$ on an open subset $U \subset S^n$, then $f*\varphi_\epsilon$ uniformly converge  
to $f$ in $C^{m}(K)$ on any compact subset $K$ of $U$. 
\end{proposition}

\begin{proof}  
The  proof is standard but  we include it here for completeness.  Since $\epsilon \in (0,1) $, for each $v \in S^n$, the support of $\varphi_\epsilon(v,\cdot)$ is compactly embedded in the hemisphere centered at $v$. We choose a rotation matrix $Q \in O(n)$ satisfying $Q(-\vec{e}_{n+1})=v$. We define a chart $ \xi :\rn \to S^n$  and a differential operator $\tN :C^1(\rn) \to \big(C^0(\rn)\big)^n$ by
\begin{align*}
\xi(x)=Q\big((x,-1)(1+|x|^2)^{-\frac{1}{2}}\big) \qquad \mbox{and} \qquad   \tN_i h(y) = \bd_i h (y) + y_i \sum_{j=1}^n y_j\bd_j h(y) 
\end{align*}
where $h\in C^1(\rn)$. Then, by direct computation, we have
\begin{align*}\label{eq:LTE Commutation in convolution}
\tN \big((f*\varphi_\epsilon)\circ\xi \big)(x)=\int_{\rn} \tN (f\circ \xi )(y) \, \varphi_\epsilon (\xi(x),\xi(y)) \, (1+|y|^2)^{-\frac{n+1}{2}}  dy\tag{6.1}
\end{align*}
which gives the desired result.
\end{proof}

Let  $\varphi_\epsilon$ be as in  Proposition \ref{prop:LTE Convolution} and let $\Sigma$ be  a strictly convex closed hypersurface. 
We will next show how to  regularize $\Sigma$ by convolving the support function of  its  $\eta$-envelope $\Sigma_\eta$ (which is a $C^{1,1}$ hypersurface) 
with the function $\varphi_\epsilon$. This is a standard argument which we include here for the reader's convenience. 

\begin{proposition}\label{prop:LTE Support function}
Let  $\Sigma^\eta$ denote  the  $\eta$-envelope of a convex closed hypersurface $\Sigma$ with a uniform lower bound for $\lambda_{\min}(\Sigma)(X)$. and let $S$ denote the support function of $\Sigma^\eta$. Assume $\Sigma$ encloses the origin. Then, there is a small constant $\alpha(\eta,\Sigma)>0$ such that for each $\epsilon \in (0,\alpha)$, $S*\varphi_\epsilon$ is the support function of a strictly convex smooth closed hypersurface $\Sigma^\eta_\epsilon$. In addition,  $S*\varphi_\epsilon \to S$, as $\epsilon \to 0$,   uniformly  on  $C^{1}(S^n)$, and the following holds
\begin{align*}
\liminf_{\epsilon \to 0} \lambda_{\min}(\Sigma^{\eta}_\epsilon)(X_\epsilon) \geq \lambda_{\min}^{\text{\em loc}}(\Sigma^{\eta})(X)
\end{align*}
where $\{X_\epsilon\}$ is a set of points $X_\epsilon \in \Sigma^\eta_\epsilon$ converging to $X \in \Sigma^\eta$ as $\epsilon \to 0$.
\end{proposition}

\begin{proof}
Let $\bar g_{ij}$ denote the standard metric on $S^n$ and let  $\bar \nabla$ be the connection on $S^n$ defined by $\bar g_{ij}$. 
  Notice that for a function $f:S^n\to \mathbb{R}^+$, if $\bar \nabla_i \bar \nabla_j f+f\bar g_{ij}$ is a positive definite matrix with respect to the metric $\bar g_{ij}$, then $f$ is the support function of a strictly convex hypersurface, and the eigenvalues of $\bar \nabla_i \bar \nabla_j f+f\bar g_{ij}$ are the principal radii of curvature of the hypersurface (c.f. \cite{U91InverseFlow}). 

Since $\Sigma^\eta$ is a uniformly convex hypersurface of class $C^{1,1}$, its  support function $S$ is of class $C^{1,1}(S^n)$, namely $\bar \nabla \bar \nabla$ exists almost everywhere. In addition, since the principal radii of curvature of the $\eta$-envelope are bounded from below by $\eta$, we have 
\begin{align*}\label{eq:LTE C^2 bound}
\eta\,  \bar g_{ij} \leq  \bar \nabla_i \bar \nabla_j S	+S \bar g_{ij} \leq \sup_{X\in \Sigma^\eta}\lambda_{\min}(\Sigma^\eta)^{-1}(X) \bar g_{ij}=\big(\eta+\sup_{X\in \Sigma}\lambda_{\min}(\Sigma)^{-1}(X)\big) \bar g_{ij}\tag{6.2}
\end{align*}
at points where $\bar \nabla \bar \nabla S$ exists. 

Recall the chart $\xi$ and the differential operator $\tN$ in Proposition \ref{prop:LTE Convolution}. Then, direct computations yield 
 \begin{align*}\label{eq:LTE Computations}
\bd_i = \tN_i - \frac{x_ix_j\tN_j}{1+|x|^2}, \tag{6.3}\qquad 
 \bd_i\bd_j= &\tN_i \tN_j -\frac{x_i^2}{1+|x|^2}\tN_i \tN_j-\frac{x_j^2}{1+|x|^2}\tN_i \tN_j+\frac{x_ix_jx_kx_l}{(1+|x|^2)^2}\tN_k \tN_l\\
& -\frac{x_ix_jx_k}{(1+|x|^2)^2}\tN_k-\frac{(\delta_{ij}-x_i)x_k}{1+|x|^2}\tN_k -\frac{x_j}{1+|x|^2} \tN_i.  
\end{align*}
Since $(\bar \nabla_i \bar \nabla_j f)\circ \xi =\bd_i \bd_j( f\circ \xi) -\Gamma^k_{ij}\bd_k (f\circ \xi) $ holds for  $fC^2(S^n)$, 
we have a  linear map $L_x$ satisfying
\begin{align*}
(\bar \nabla_i \bar \nabla_j f)\circ \xi(x)+\bar g_{ij} f\circ \xi(x)=L_x(\tN\tN (f\circ \xi)(x),\tN (f\circ \xi)(x),f\circ \xi(x)). 
\end{align*}
For convenience, we denote $f\circ \xi$ and  $L_x(\tN\tN f,\tN f, f)$ by $f$ and $L_x(f)$, respectively. Then, since $S$ is a.e. second order differentiable, \eqref{eq:LTE Commutation in convolution} gives
\begin{align*}\label{eq:LTE convolution for linear operator}
&(\bar \nabla_i \bar \nabla_j S*\varphi_\epsilon+\bar g_{ij} S*\varphi_\epsilon)(x)=L_x(S*\varphi_\epsilon(x))=\int_{\rn} L_x (S(y))\varphi_\epsilon (x,y)\frac{dy}{(1+|y|^2)^{\frac{n+1}{2}}}\\
&= \int_{\rn} (\bar \nabla_i \bar \nabla_j S+\bar g_{ij} S)(y)\frac{\varphi_\epsilon (x,y)dy}{(1+|y|^2)^{\frac{n+1}{2}}} +\int_{\rn} \big(L_x(S(y))-L_y (S(y))\big)\frac{\varphi_\epsilon (x,y)dy}{(1+|y|^2)^{\frac{n+1}{2}}}\tag{6.4}\\
&\, \geq \eta\, \bar g_{ij}-|L_x(S(y))-L_y (S(y))|\bar g_{ij}. 
\end{align*}
Notice that $S*\varphi_\epsilon \to S$ uniformly on $C^1(S^n)$ by Proposition   \ref{prop:LTE Convolution}, and  $|\bd^2 S|$ is bounded by \eqref{eq:LTE C^2 bound} and \eqref{eq:LTE Computations}. Also, 
we have $L_x \to  L_y$ as $x \to y$ by \eqref{eq:LTE Computations}. Hence, $|L_x(S(y))-L_y (S(y))|$ converges to zero. Hence,
\begin{align*}
\liminf_{\epsilon \to 0} (\bar \nabla_i \bar \nabla_j S*\varphi_\epsilon+\bar g_{ij} S*\varphi_\epsilon)(x) \geq \eta\, \bar g_{ij}. 
\end{align*}
So, there exist some $\alpha$ such that for each  $\epsilon \in (0,\alpha)$, there is a strictly convex hypersurface $\Sigma^\epsilon_\epsilon$ whose the support function is $S*\varphi_\epsilon$. Similarly, we can derive from \eqref{eq:LTE convolution for linear operator} that 
\begin{align*}
\limsup_{\epsilon \to 0} (\bar \nabla_i \bar \nabla_j S*\varphi_\epsilon+\bar g_{ij} S*\varphi_\epsilon)(x_\epsilon) \leq \big(\lambda_{\min}^{\text{loc}}(\Sigma^{\eta})(X)\big)^{-1}\bar g_{ij}
\end{align*}
where $x_\epsilon$ converge to a point $x$ such that  $\xi(x)$ is the outer normal vector  of $\Sigma^\eta$ at $X$.
\end{proof}

We will now give the proof of our long time  existence result,  Theorem \ref{thm:INT Existence}. 

\begin{proof}[Proof of Theorem {\em  \ref{thm:INT Existence}}] Let $u_0, \Sigma_0$ and $\Omega_0$ be as in the statement 	
of the theorem and assume without loss of generality   that $\displaystyle \inf_{\Omega} u_0 =0$.
We will obtain a solution $\Sigma_t\coloneqq \{ (x,    u(\cdot,t)):  x \in \Omega_t\subset \rn\}$
as a limit  $$  \Sigma_t\coloneqq \lim_{j \to +\infty}   \Sj^j_t$$   where  $\Sj_t^j$ is 
 a strictly convex closed hypersurface which is symmetric with respect to the hyperplane $x_{n+1}=j$
and also evolves by the $Q_k$-flow \eqref{eq:INT Qkflow}. Then, the symmetry guarantees that the lower half of $\Sj^j_t$ is a graph. Thus, by applying the local a'priori estimates shown in sections
\ref{sec-prelim}-\ref{sec-higher} on compact subsets of $\rno$, we have 
the uniform $C^{\infty}$ bounds on the lower half of $\Sj^j_t$ necessary to
pass to the limit. 

Now, given the initial data $u_0$, $\Sigma_0$, and $\Omega_0$, we define $\Sj^j_t$ and $\Sigma_t$
as follows. 

\smallskip

\noindent{\em Step 1 : The construction of the approximating sequence $\Sj_t^j$.} \,   Let  $u_0, \Sigma_0$ and $\Omega_0$ be as in  
Theorem \ref{thm:INT Existence}  and assume that $\displaystyle \inf_{\Omega} u_0 =u_0(0)= 0$.  
Since $\Sigma_0$ is not necessarily strictly convex we consider the   strictly convex rotationally symmetric non-negative entire function $\varphi(x)=\rn \to \mathbb{R}^+_0$ of class $C^{\infty}(\rn)$ defined by
\begin{align*}
\varphi(x)=  \int^{|x|}_0 \arctan r \,  dr 
\end{align*}
and   for each $j \in \mathbb{N}$, we define the approximate strictly convex smooth function $u_j:\Omega_0 \to \mathbb{R}$ with the corresponding graph $\widetilde{\Sigma}^j_0$ by
\begin{align*}
\tilde{u}_j(x)=u_0(x)+\varphi(x)/j \,, \quad \widetilde{\Sigma}^j_0=\{(x,\tilde{u}_j(x)):x\in \Omega_0\}. 
\end{align*}
Then, we  reflect $\widetilde{\Sigma}^j_0 \bigcap \big(\rn \times [0,j]\big)$
 over the $j$-level hyperplane $\rn \times \{j\}\coloneqq\{(x,j):x \in \rn\}$ to  obtain a strictly convex closed hypersurface $\widetilde{\Sj}^j_0$ defined by
\begin{align*}
\widetilde{\Sj}^j_0 = \big\{(x,h) \in \rno :x \in \Omega_0,\; \tilde{u}_j(x)\leq j,\;  h \in \{\tilde{u}_j(x),2j-\tilde{u}_j(x)\}\big\}. 
\end{align*}
Since $\widetilde{\Sj}^j_0$ fails to be smooth at its intersection with the hyperplane $\rn \times \{j\}$, we again approximate $\tilde{\Sj}^j_0$ by a strictly convex closed $C^{1,1}$ hypersurface $\bar \Sj^j_0$ which is the $(1/j)$-envelope of $\widetilde{\Sj}^j_0$, namely $\bar \Sj^j_0 \coloneq (\widetilde{\Sj}^j_0)^{1/j} $.
Then, we denote by $\bar \Sj^j_t$ the unique closed solution of the Mean curvature flow ($*^n_1$)  defined for $t \in (0,\bar T_j)$
satisfying $\displaystyle \lim_{t \rightarrow 0} \bar  \Sj^j_t= \bar \Sj^j_0$, where $\bar T_j$ is its maximal existing time. In the next step, we will show
\begin{enumerate}[label=(\subscript{P}{\arabic*})]
\item $\bar \Sj^j_t$ is symmetric  with respect to the $j$-level hyperplane $\rn \times \{j\}$.
\item $\bar \Sj^j_t$ is a strictly convex smooth closed solution for $t \in (0,\bar T_j)$.
\item $\bar \Sj^j_t$ is a smooth graph solution on  $\rn \times [-1,j-1/j]$ for $t\in [0,\bar T_j)$.
\end{enumerate} 
Observe that by choosing  $j \gg 1$ (depending on $\Sigma_0$ and
$\Omega_0$) we may assume that the maximal existing time $\bar T_j$ satisfies 
 $\bar{T}_j > 1/j$ so that we have a strictly convex smooth closed hypersurface $\Sj^j_0 \coloneqq \bar \Sj^{j}_{1/j}$. We denote by $\Sj^j_t$ the unique strictly convex smooth closed solution of (\ref{eq:INT Qkflow}) defined for $t \in [0,T_j)$, where $T_j$ is its  maximal existing time (c.f. in \cite{A94NonlinearFlow}).   Finally, we set
$$\Sigma_t =  \bd \big \{ \bigcup_{j\in \mathbb{N}}\cv(\Gamma^j_t)\big \}, \quad 
 T=\sup_{j \in \mathbb{N}} T_j.$$

\medskip

\noindent{\em Step 2 : Proof of the properties} ($P_1$) - ($P_3$). By the well-known local regularity result in \cite{White05Local}, there is a smooth solution $\bar \Sj^j_t$ of the MCF defined for $t \in (0,\bar T_j)$ satisfying $\displaystyle \lim_{t \to 0}\bar \Sj^j_t =\bar \Sj^j_0$. Since $\bar \Sj^j_0$ is symmetric  with respect to $\rn \times \{j\}$, the uniqueness guarantees ($P_1$).

Now, we consider $(0,j)$ as the origin and let $S^j_0$ denote the support function of $\bar \Sj^j_0$ with respect to out new origin $(0,j)$. We define $S^{j,\epsilon}_0$ by the convolution $S^{j,\epsilon}_0=S^j_0*\varphi_\epsilon$ with the mollifier $\varphi_\epsilon$ given in Proposition  \ref{prop:LTE Convolution}. By Proposition \ref{prop:LTE Support function}, for each small $\epsilon \ll 1$, there is a strictly convex smooth closed hypersurface $\bar \Sj^{j,\epsilon}_0$ whose support function is $S^{j,\epsilon}_0$. So, we have a strictly convex smooth solution $\bar \Sj^{j,\epsilon}_t$ of the MCF defined for $t \in [0,\bar T_{j,\epsilon})$. In addition, since $\bar \Sj^{j,\epsilon}_0$ have a uniform lower bound of principal curvatures by Proposition \ref{prop:LTE Support function},  $\bar \Sj^{j,\epsilon}_t$ also have the uniform bound. Thus, the limit $\bar \Sj^j_t$ is a strictly convex hypersurface, which implies ($P_2$).

Observe that ($P_1$) and ($P_2$) prove that $\bar \Sj^{j,\epsilon}_t$ is a graph in $\rn \times [-1,j)$. Thus, we can apply the a'priori estimates in previous sections for $\bar \Sj^{j,\epsilon}_t$ with any cut-off function $\psi\coloneq (M-\bar u(p,t))_+$ of $M<j$. Since $\bar \Sj^j_0$ is smooth in $\rn \times [-1,j-1/j]$, Proposition \ref{prop:LTE Convolution} gives the $C^{\infty}$ convergence of $S^{j,\epsilon}_0$ to $S^j_0$ in $U_j$, where $ U_j \subset S^n$ is the set of outer unit normal vectors of $\bar \Sj^j_0 \bigcap \big(\rn \times [-1,j-1/j] \big)$. Hence, we have local uniform estimates for $\bar \Sj^j_t$ in $\rn \times [-1,j-1/j)$ up to $t=0$, which yields ($P_3$).

\bigskip

\noindent {\em Step 3 : Passing $\Sj^j_t$ to the limit $\Sigma_t$.}  First we have $\bar \Sj^j_0 \preceq \bar \Sj^{j+1}_0$ by definition.  Hence, the  comparison principle 
gives that $\Sj^j_0\coloneq \bar \Sj^j_{1/j}\prec \bar \Sj^{j+1}_{1/(j+1)}\coloneq \Sj^{j+1}_0$ and  as a consequence also  $\Sigma^j_t \prec \Sigma^{j+1}_t \prec \Sigma_0$. Therefore, $ \bigcup_{j\in \mathbb{N}}\cv(\Gamma^j_t)$ is a convex body and $\Sigma_t$ is a complete and convex hypersurface in $\rno$.

 Next we notice that by ($P_3$), we can apply the estimates in sections  \ref{sec-prelim}-\ref{sec-higher}  for the solution $\bar \Sj^j_t$
 of the Mean curvature flow ($*^n_1$).   We choose any constant $M_0 >0$ and $\epsilon\in (0,1)$, and for  $j \gg  M_0+1$, we apply the estimates in the following order
\begin{enumerate}
\item Theorem \ref{thm:Pre Gradient} and Corollary \ref{cor:Pre Lower Bound of Q_k} for $\{\bar \Sj^j_t\}_{t \in [0,1/j]}$ with $M=M_0$.
\item Theorem \ref{thm:SE Speed Estimate} for $\{\bar \Sj^j_t\}_{t \in [0,1/j]}$ with $M=M_0-\epsilon$.
\item Theorem \ref{thm:CE Local Pinching Estimate} for $\{\bar \Sj^j_t\}_{t \in [0,1/j]}$ with $M=M_0-2\epsilon$.
\item For each $m \in \mathbb{N}$, Theorem \ref{thm:DE Higher Order Estimates} for $\{\bar \Sj^j_t\}_{t \in [0,1/j]}$ with $M=M_0-(3-1/2^m)\epsilon$.
\end{enumerate}
Thus, we  obtain  uniform bounds on  $\nu$, $Q_k^{-1}$, and $|\nabla^m A|^2$ for all $m \in \mathbb{N}\bigcup \{0\}$ for $\bar \Sj^j_{1/j}\coloneq \Sj^j_0$ with  $j \geq M_0+1$ on  $\rn \times [0,M_0-3\epsilon]$. 

Recall now that  $\Gamma_t^j$ denotes the solution of  \eqref{eq:INT Qkflow} with initial data $\Gamma_0^j$. 
We can  then apply in a similar manner as above the estimates in Theorem \ref{thm:Pre Gradient}, 
Theorem \ref{thm:SE Speed Estimate},  Theorem \ref{thm:CE Local Pinching Estimate}, and Theorem \ref{thm:DE Higher Order Estimates} 
to obtain uniform bounds on $\nu$, $Q_k^{-1}$, and $|\nabla^m A|^2$ which old on $\rn \times [0,M_0-4\epsilon]$
and for $j \gg M_0+1$. Thus, the limit $\Sigma_t$ is also a smooth solution of \eqref{eq:INT Qkflow} with the positive $Q_k$ curvature. In addition, the local upper bound of $\nu$ implies that each $\Sigma_t$ is a graph.

\bigskip

\noindent{\em Step 4 : Lower bound on  the existence time  $T$.} We will end our proof by estimating from below the maximal time of existence $T$.
Recall the strictly convex perturbation     $\widetilde{\Sigma}_j =\{(x,\tilde{u}_j(x)):x\in \Omega_0\}$ 
of our initial graph $\Sigma_0$ as defined in Step 1 of the proof. Since we have $B_R(x_0)\subset \Omega_0$, for each $\epsilon \in (0,R/2)$, there is a constant $d_\epsilon$ such that $B_{R-\epsilon}^{n+1}((x_0,d_\epsilon))\prec \widetilde{\Sigma}_1 \preceq \widetilde{\Sigma}_j$. Then,   $B_{R-\epsilon}^{n+1}((x_0,d_\epsilon)) \prec \bar \Sj^j_0$
holds  for all $j \geq d_\epsilon$. 

On the the hand, $\bd B^{n+1}_{\bar \rho_\epsilon(t)}((x_0,h_\epsilon))$ is a solution of the MCF, where $\bar \rho_\epsilon(t)=((R-\epsilon)^2-2nt)^{\frac{1}{2}}$. Hence, for each $j > \max\{d_\epsilon,R^2/(8n)\}$, the comparison principle implies that  $\bd B^{n+1}_{\sqrt{(R-\epsilon)^2-2n/j}} \prec \bar\Sj^j_{1/j}\coloneq \Sj^j_0$.
Also,  $\bd B^{n+1}_{\rho^j_\epsilon(t)}((x_0,d_\epsilon))$, with $\rho^j_\epsilon(t)=(R-\epsilon)^2-2n/j-2((n-k+1)/k)t$,  is a solution of the $Q_k$-flow, and therefore the comparison principle leads to $B^{n+1}_{\rho^j_\epsilon(t)}\prec \Sj^j_t$. Since $\Sj^j_t$ exists while $\rho^j_\epsilon(t)>0$, we have $$T \geq ((R-\epsilon)^2-2n/j)/(2(n-k+1)/k).$$
By passing $j \to \infty$ and $\epsilon \to 0$, we  obtain the desired bound  $T \geq \frac{k}{2(n-k+1)}R^2$.
In particular, if there is a ball $B_R(x_R) \subset \Omega_0$ for each $R$, then  $T=+\infty$.
\end{proof}

\centerline{\bf Acknowledgements}

\smallskip 

\noindent P. Daskalopoulos  has been partially supported by NSF grant DMS-1266172.

\bibliographystyle{abbrv}

\bibliography{myref}

\end{document}